\documentclass[11pt]{amsart}
\usepackage[margin=1in]{geometry}

\usepackage{amssymb}
\usepackage{amsthm}
\usepackage{amsmath}
\usepackage{mathrsfs}
\usepackage{amsbsy}
\usepackage[all]{xy}
\usepackage{bm}
\usepackage{hyperref}
\usepackage{tikz}
\usepackage{array}
\usepackage{float}
\usepackage{enumerate}
\usepackage{xcolor}
\usepackage{hhline}
\allowdisplaybreaks
\usepackage[noadjust]{cite}
\usepackage{ytableau}
\usepackage{shuffle}

\usepackage[noabbrev,capitalise,nameinlink]{cleveref}

\definecolor{lavender}{rgb}{0.4,0,1.0}

\hypersetup{colorlinks=true, citecolor=lavender, linkcolor=lavender,urlcolor=lavender}

\makeatletter
\let\save@mathaccent\mathaccent
\newcommand*\if@single[3]{%
  \setbox0\hbox{${\mathaccent"0362{#1}}^H$}%
  \setbox2\hbox{${\mathaccent"0362{\kern0pt#1}}^H$}%
  \ifdim\ht0=\ht2 #3\else #2\fi
  }
\newcommand*\rel@kern[1]{\kern#1\dimexpr\macc@kerna}
\newcommand*\widebar[1]{\@ifnextchar^{{\wide@bar{#1}{0}}}{\wide@bar{#1}{1}}}
\newcommand*\wide@bar[2]{\if@single{#1}{\wide@bar@{#1}{#2}{1}}{\wide@bar@{#1}{#2}{2}}}
\newcommand*\wide@bar@[3]{%
  \begingroup
  \def\mathaccent##1##2{%
    \let\mathaccent\save@mathaccent
    \if#32 \let\macc@nucleus\first@char \fi
    \setbox\z@\hbox{$\macc@style{\macc@nucleus}_{}$}%
    \setbox\tw@\hbox{$\macc@style{\macc@nucleus}{}_{}$}%
    \dimen@\wd\tw@
    \advance\dimen@-\wd\z@
    \divide\dimen@ 3
    \@tempdima\wd\tw@
    \advance\@tempdima-\scriptspace
    \divide\@tempdima 10
    \advance\dimen@-\@tempdima
    \ifdim\dimen@>\z@ \dimen@0pt\fi
    \rel@kern{0.6}\kern-\dimen@
    \if#31
      \overline{\rel@kern{-0.6}\kern\dimen@\macc@nucleus\rel@kern{0.4}\kern\dimen@}%
      \advance\dimen@0.4\dimexpr\macc@kerna
      \let\final@kern#2%
      \ifdim\dimen@<\z@ \let\final@kern1\fi
      \if\final@kern1 \kern-\dimen@\fi
    \else
      \overline{\rel@kern{-0.6}\kern\dimen@#1}%
    \fi
  }%
  \macc@depth\@ne
  \let\math@bgroup\@empty \let\math@egroup\macc@set@skewchar
  \mathsurround\z@ \frozen@everymath{\mathgroup\macc@group\relax}%
  \macc@set@skewchar\relax
  \let\mathaccentV\macc@nested@a
  \if#31
    \macc@nested@a\relax111{#1}%
  \else
    \def\gobble@till@marker##1\endmarker{}%
    \futurelet\first@char\gobble@till@marker#1\endmarker
    \ifcat\noexpand\first@char A\else
      \def\first@char{}%
    \fi
    \macc@nested@a\relax111{\first@char}%
  \fi
  \endgroup
}
\makeatother

\usepackage{caption}
\usepackage{tabu}
\usepackage{diagbox}

\newenvironment{enumerate*}%
  {\begin{enumerate}[(I)]%
    \setlength{\itemsep}{10pt}%
    \setlength{\parskip}{0pt}}%
  {\end{enumerate}}

\newtheorem{theorem}{Theorem}[section]
\newtheorem{proposition}[theorem]{Proposition}
\newtheorem{corollary}[theorem]{Corollary}

\newtheorem{lemma}[theorem]{Lemma}

\theoremstyle{definition}

\newtheorem{remark}[theorem]{Remark}
\newtheorem{example}[theorem]{Example}

\newcommand{\Inv}{\mathrm{Inv}}
\newcommand{\LL}{{\mathsf{L}}}
\newcommand{\DD}{\mathbf{L}}
\newcommand{\End}{\mathrm{End}}
\newcommand{\demark}{\mathrm{demark}}
\newcommand{\dom}{\mathrm{dom}}
\newcommand{\ch}{\mathrm{ch}}
\newcommand{\Res}{\mathrm{Res}}
\newcommand{\Ind}{\mathrm{Ind}}

\newcommand{\QSym}{\ensuremath{\mathrm{QSym}}}

\newcommand{\Peak}{\ensuremath{\mathrm{Peak}}}
\newcommand{\Val}{\ensuremath{\mathrm{Val}}}
\newcommand{\stand}{\ensuremath{\mathrm{std}}}

\newcommand{\Des}{\ensuremath{\mathrm{Des}}}

\newcommand{\SDT}{\ensuremath{\mathsf{SDT}}}
\newcommand{\SSShDT}{\ensuremath{\mathsf{SSShDT}}}
\newcommand{\SShDT}{\ensuremath{\mathsf{SShDT}}}
\newcommand{\MSShDT}{\ensuremath{\mathsf{MSShDT}}}
\newcommand{\wt}{\ensuremath{\mathrm{wt}}}

\newcommand{\dfn}[1]{\textcolor{blue}{\emph{#1}}}

\usepackage{etoolbox}

\begin{document}

\title{0-Hecke Modules, Domino Tableaux, and Type-$B$ Quasisymmetric Functions}
\subjclass[2010]{}

\author[Colin Defant]{Colin Defant}
\address[]{Department of Mathematics, Harvard University, Cambridge, MA 02139, USA}
\email{colindefant@gmail.com}

\author[Dominic Searles]{Dominic Searles}
\address[]{Department of Mathematics and Statistics, University of Otago, Dunedin 9016, New Zealand}
\email{dominic.searles@otago.ac.nz}

\maketitle

\begin{abstract}
We extend the notion of ascent-compatibility from symmetric groups to all Coxeter groups, thereby providing a type-independent framework for constructing families of modules of $0$-Hecke algebras. We apply this framework in type $B$ to give representation-theoretic interpretations of a number of noteworthy families of type-$B$ quasisymmetric functions. Next, we construct modules of the type-$B$ $0$-Hecke algebra corresponding to type-$B$ analogues of Schur functions and introduce a type-$B$ analogue of Schur $Q$-functions; we prove that these \emph{shifted domino functions} expand positively in the type-$B$ peak functions. We define a type-$B$ analogue of the $0$-Hecke--Clifford algebra, and we use this to provide representation-theoretic interpretations for both the type-$B$ peak functions and the shifted domino functions. 
We consider the modules of this algebra induced from type-$B$ $0$-Hecke modules constructed via ascent-compatibility and prove a general formula, in terms of type-$B$ peak functions, for the type-$B$ quasisymmetric characteristics of the restrictions of these modules.
\end{abstract}

\section{Introduction}\label{sec:intro}

The 0-Hecke algebra $H_W(0)$ of a Coxeter system $(W,S)$ is the $q=0$ specialization of the Hecke algebra $H_W(q)$ of $W$. In type $A$, representations of $0$-Hecke algebras can be interpreted in terms of the Hopf algebra $\QSym$ of quasisymmetric functions; specifically, the \emph{quasisymmetric characteristic} map \cite{DKLT} defines an isomorphism between the Grothendieck group of finite-dimensional $0$-Hecke modules and $\QSym$. Much recent work, including \cite{Bardwell.Searles, BBSSZ, NSvWVW:0Hecke, Searles:0Hecke, TvW:1}, has been devoted to giving representation-theoretic interpretations of important bases of quasisymmetric functions by constructing families of type-$A$ 0-Hecke modules that are associated to these bases via the quasisymmetric characteristic map. 

Recently, a uniform method for constructing type-$A$ $0$-Hecke modules was introduced in \cite{Searles:0HC}, based on an \emph{ascent-compatibility} condition defined on subsets of symmetric groups. As far as we are aware, all of the $0$-Hecke modules that have been constructed for notable families of quasisymmetric functions can be obtained via this ascent-compatibility framework.  
Our first main result, which we present in \cref{sec:ascentcompatible}, consists of a type-independent definition of ascent-compatibility for all (not necessarily finite) Coxeter groups together with a consequent extension of the construction method for $0$-Hecke modules from \cite{Searles:0Hecke} from type~$A$ to all Coxeter groups. We show that any convex subset of the left weak order that has a unique maximal element is ascent-compatible, allowing the construction of large families of $0$-Hecke modules for any Coxeter system.

We then specialize our attention to the type $B$ case. There is a type-$B$ analogue of the type-$A$ connection between $0$-Hecke algebras and quasisymmetric functions: Chow \cite{Chow} introduced the ring $\QSym^B$ of type-$B$ quasisymmetric functions, and Huang \cite{Huang} defined a quasisymmetric characteristic map from the Grothendieck group of finite-dimensional type-$B$ 0-Hecke modules to $\QSym^B$.  
Consequently, ascent-compatible subsets of hyperoctahedral groups give rise to families of type-$B$ quasisymmetric functions that have representation-theoretic significance. 
As an application of our methods, we construct type-$B$ 0-Hecke modules corresponding to families of type-$B$ quasisymmetric functions defined by Mayorova and Vassilieva \cite{MV} that exhibit a type-$B$ analogue of Schur positivity.

In \cref{sec:domino}, we first consider 
the \emph{domino functions} introduced by Mayorova and Vassilieva \cite{MV}. These type-$B$ analogues of Schur functions are defined in terms of \emph{domino tableaux} of partition shape. 
We give a representation-theoretic interpretation of these functions by defining an action of the type-$B$ $0$-Hecke algebra on standard domino tableaux of a given shape: the quasisymmetric characteristics of the resulting type-$B$ $0$-Hecke modules are precisely the type-$B$ Schur functions. 
We then consider the semistandard domino tableaux of shifted partition shape introduced by Chemli \cite{Chemli} and used to provide an expansion formula for products of Schur $Q$-functions.
We define a variant of these semistandard domino tableaux whose generating functions are in $\QSym^B$. These functions, which we call \emph{shifted domino functions}, form type-$B$ analogues of Schur $Q$-functions. We prove that shifted domino functions expand positively in the type-$B$ peak functions introduced by Petersen \cite{Petersen}. Our expansion is indexed by type-$B$ peak sets of standard shifted domino tableaux, analogously to how peak sets of standard shifted tableaux index the expansion of Schur $Q$-functions in (type-$A$) peak functions.

In type~$A$, the \emph{$0$-Hecke--Clifford algebras} are obtained by combining the $0$-Hecke algebras and the Clifford algebras. Bergeron, Hivert, and Thibon defined a quasisymmetric characteristic map \cite{BHT} from the Grothendieck group of finite-dimensional $0$-Hecke--Clifford modules to the type-$A$ peak algebra of Stembridge \cite{Stembridge:enriched}. Type-$A$ ascent-compatibility was applied in \cite{Searles:0HC} to construct $0$-Hecke--Clifford modules whose quasisymmetric characteristics are certain bases of the peak algebra. In \cref{sec:0HeckeClifford}, we introduce a type-$B$ analogue of the $0$-Hecke--Clifford algebras and show that the isomorphism classes of the modules of these algebras induced from the simple type-$B$ $0$-Hecke modules are indexed by the type-$B$ peak sets. Moreover, we show that the quasisymmetric characteristics of the restrictions of these modules to the type-$B$ $0$-Hecke algebra are precisely the type-$B$ peak functions of Petersen \cite{Petersen}, thereby giving representation-theoretic significance to these functions. 

In \cref{sec:CliffordApplications}, we prove a general result that allows us to compute the type-$B$ quasisymmetric characteristic of the restriction of an induction of a type-$B$ $0$-Hecke module. This result applies whenever the module is obtained via the type-$B$ ascent compatibility framework. We also apply it in order to give representation-theoretic significance to our aforementioned shifted domino functions.

\section{Preliminaries}

\subsection{Quasisymmetric functions}\label{subsec:quasisymmetric} 

We use the notation $[n]=\{1,\ldots,n\}$ and $[0,n]=\{0,1,\ldots,n\}$. A \dfn{composition} is a finite tuple of positive integers. If $\alpha=(\alpha_1,\ldots,\alpha_k)$ is a composition such that $\alpha_1+\cdots+\alpha_k=n$, then we say $\alpha$ is a composition of $n$ and write $\alpha\vDash n$; the integers $\alpha_i$ are called the \dfn{parts} of $\alpha$. If $\alpha_1\geq\cdots\geq\alpha_k$, then we say $\alpha$ is a \dfn{partition} of $n$ and write $\alpha\vdash n$. The \dfn{descent set} of $\alpha$ is the set $\Des(\alpha)=\{\alpha_1, \alpha_1+\alpha_2, \ldots, \alpha_1 + \alpha_2 + \cdots + \alpha_{k-1}\}$. The map $\alpha\mapsto\Des(\alpha)$ is a bijection from the set of compositions of $n$ to the set of subsets of $[n-1]$.

For $\alpha\vDash n$, the \dfn{fundamental quasisymmetric function} $F_\alpha$ is defined by
\[F_\alpha = \sum_{\substack{1\le i_1 \le i_2 \le \cdots \le i_n\\ j\in \Des(\alpha) \implies i_j<i_{j+1}}}x_{i_1}x_{i_2}\cdots x_{i_n}.\]
As $\alpha$ ranges over all compositions, the fundamental quasisymmetric functions $F_{\alpha}$ form a linear basis of the algebra $\QSym$ of quasisymmetric functions. It will be notationally simpler for us to index fundamental quasisymmetric functions (and their variants) by sets instead of compositions. Thus, we will write \[F_I = \sum_{\substack{1\le i_1 \le i_2 \le \cdots \le i_n\\ j\in I \implies i_j<i_{j+1}}}x_{i_1}x_{i_2}\cdots x_{i_n}.\] Note that this definition depends implicitly on $n$; this should not lead to confusion since $n$ will often be fixed. 

\begin{example}
Let $n=4$. The composition $\alpha=(2,1,1)$ has descent set $\{2,3\}$, so \[F_{\alpha}=F_{\{2,3\}}=\displaystyle\sum_{1\leq i_1\leq i_2<i_3<i_4}x_{i_1}x_{i_2}x_{i_3}x_{i_4}.\]
\end{example}

Chow \cite{Chow} introduced type-$B$ analogues of the fundamental quasisymmetric functions that involve an additional variable $x_0$. A \dfn{type-$B$ composition} of $n$ is a tuple $(\alpha_1,\ldots,\alpha_k)$ of integers such that $\alpha_1\geq 0$ and $\alpha_2,\ldots,\alpha_k\geq 1$. As for ordinary (type-$A$) compositions, the descent set of $\alpha$ is the set $\Des(\alpha)=\{\alpha_1, \alpha_1+\alpha_2, \ldots, \alpha_1 + \alpha_2 + \cdots + \alpha_{k-1}\}$. The \dfn{type-$B$ fundamental quasisymmetric function} $F_\alpha^B$ is defined by
\[F^B_\alpha = \sum_{\substack{0=i_0\le i_1 \le \cdots \le i_n\\ j\in \Des(\alpha) \implies i_j<i_{j+1}}}x_{i_1}x_{i_2}\cdots x_{i_n}.\] Note that $x_0$ appears in $F^B_\alpha$ if and only if $0$ is not a part of $\alpha$. As before, we will often abuse notation and index type-$B$ quasisymmetric functions with sets; we write $F_{\Des(\alpha)}^B$ instead of $F_\alpha^B$. Again, this notation suppresses the implicit dependence on $n$. 

\begin{example}
Let $n=4$. Then \[F_{(2,1,1)}^B=F_{\{2,3\}}^B=\sum_{0\leq i_1\leq i_2<i_3<i_4}x_{i_1}x_{i_2}x_{i_3}x_{i_4}\quad\text{and}\quad F_{(0,3,1)}^B=F_{\{0,3\}}^B=\sum_{0<i_1\leq i_2\leq i_3<i_4}x_{i_1}x_{i_2}x_{i_3}x_{i_4}.\]
\end{example}

A subset $P\subseteq [n-1]$ is called a \dfn{peak set} if $1\notin P$ and $P$ does not contain two consecutive integers. The compositions whose descent sets are peak sets are called \dfn{peak compositions}; these are precisely the compositions whose non-final parts are all greater than $1$. Given a peak set $P\subseteq [n-1]$, the \dfn{peak function} $K_P$ is defined as
\[K_P = 2^{|P|+1}\sum_{\substack{J\subseteq [n-1] \\ P\subseteq J\triangle (J+1)}} F_J,\] 
where $J+1 = \{j+1 : j\in J\}$ and $\triangle$ denotes symmetric difference.

Petersen introduced type-$B$ analogues of the peak functions in \cite{Petersen}. A subset $P\subseteq [n-1]$ is called a \dfn{type-$B$ peak set} if $P$ contains no two consecutive integers; note that a type-$B$ peak set is permitted to contain $1$. For each type-$B$ peak set $P$, there are two type-$B$ peak functions $K_{(0,P)}$ and $K_{(1,P)}$ defined by
\[K_{(0,P)} = 2^{|P|}\sum_{\substack{J\subseteq [0,n-1] \\ P\subseteq J\triangle (J+1)}} F^B_J \quad \mbox{and} \quad K_{(1,P)} = 2^{|P|+1}\sum_{\substack{0\in J \subseteq [0,n-1] \\ P\subseteq J\triangle (J+1)}}F^B_J,\]
with the proviso that $K_{(1,P)}$ is defined only if $1\notin P$. 
The type-$B$ peak set associated to a set $I\subseteq[0,n-1]$ is \[\Peak^B(I)=\{p\in [n-1]: p\in I \text{ and }p-1\not\in I\}.\] Let us define \begin{equation}\label{eq:Delta}
\Delta^B(I)=K_{(\zeta(I),\Peak^B(I))},
\end{equation} 
where \[\zeta(I)=\begin{cases} 0 & \mbox{ if } 0\not\in I \\
1 & \mbox{ if } 0\in I.
\end{cases}\]
We will also make use of the \dfn{type-$B$ valley set} of $I$, which we define to be \[\Val^B(I)=\{v\in[n]:v\notin I\text{ and }v-1\in I\}.\] Note that \begin{equation}\label{eq:PeakVal}
\left\lvert\Val^B(I)\right\rvert=\left\lvert\Peak^B(I)\right\rvert+\zeta(I).
\end{equation}

\subsection{0-Hecke algebras and modules}
Given symbols $\gamma,\delta$ and a nonnegative integer $r$, we write $[\gamma\vert\delta]_r$ for the word \[\underbrace{\cdots\gamma\delta}_r\] of length $r$ that alternates between $\gamma$ and $\delta$ and ends with $\delta$. For example, $[\gamma\vert\delta]_4=\gamma\delta\gamma\delta$, while $[\gamma\vert\delta]_5=\delta\gamma\delta\gamma\delta$. 

Let $(W,S)$ be a Coxeter system. Thus, $W$ is generated by $S$, and each element of $S$ is an involution. For all distinct $s,t\in S$, there is an integer $m(s,t)=m(t,s)\in\{2,3,\ldots\}\cup\{\infty\}$ such that $[s\vert t]_{m(s,t)}=[t\vert s]_{m(s,t)}$ (this relation is not present if $m(s,t)=\infty$). The symmetric group $S_n$ has the simple generating set $\{s_1,\ldots,s_{n-1}\}$, where $m(s_i,s_{i+1})=3$ for all $1\leq i\leq n-2$ and $m(s_i,s_j)=2$ whenever $|i-j|\geq 2$. The hyperoctahedral group $B_n$ has the simple generating set $\{s_0,s_1,\ldots,s_{n-1}\}$, where $m(s_0,s_1)=4$, $m(s_0,s_j)=2$ for all $j\geq 2$, and $m(s_i,s_j)$ is the same as for the symmetric group when $i,j\geq 1$.

The \dfn{$0$-Hecke algebra} $H_W(0)$ of the Coxeter system $(W,S)$ is a certain deformation of the group algebra of $W$; it has a generating set $\{\pi_s : s\in S\}$ satisfying the relations \[\pi_s^2=-\pi_s\quad\text{and}\quad [\pi_s\vert\pi_t]_{m(s,t)} = [\pi_t\vert\pi_s]_{m(s,t)}\] for all distinct $s,t\in S$. 

We will mostly focus on $0$-Hecke algebras when $W$ is 
$S_n$ or 
$B_n$. These are referred to as, respectively, the type-$A$ $0$-Hecke algebra $H_n(0)$ and the type-$B$ $0$-Hecke algebra $H_n^B(0)$. The type-$A$ $0$-Hecke 
algebra $H_n(0)$ is generated by $\{\pi_1, \ldots , \pi_{n-1}\}$ satisfying the relations 
\begin{align}
\pi_i^2 & = -\pi_i;\label{eq:Hn(0)_relations1}  \\
\pi_i\pi_j & =  \pi_j\pi_i  \,\,\,\, \mbox{ if } |i-j|\ge 2; \label{eq:Hn(0)_relations2}\\
\pi_i\pi_{i+1}\pi_i & = \pi_{i+1}\pi_i\pi_{i+1}, \label{eq:Hn(0)_relations3}
\end{align}
while the type-$B$ $0$-Hecke algebra $H_n^B(0)$ is generated by $\{\pi_0, \pi_1, \ldots , \pi_{n-1}\}$ with the same relations \eqref{eq:Hn(0)_relations1}, \eqref{eq:Hn(0)_relations2}, \eqref{eq:Hn(0)_relations3} for $1\le i \le n-1$ and the additional relations
\begin{align}
\pi_0^2 & = -\pi_0; \label{eq:Hn(0)_relations4} \\
\pi_0\pi_i & =  \pi_i\pi_0  \,\,\,\, \mbox{ if } i\ge 2; \label{eq:Hn(0)_relations5}\\
\pi_0\pi_1\pi_0\pi_1 & = \pi_1\pi_0\pi_1\pi_0.\label{eq:Hn(0)_relations6}
\end{align}
By \cite{Norton}, the simple modules of $H_n(0)$ (respectively, $H_n^B(0)$) are all one-dimensional, and they are in bijection with the subsets of the set of simple generators of $S_n$ (respectively, $B_n$). 
For $I\subseteq[n-1]$, let us (slightly abusing notation) write $\mathcal S_I$ for the simple $H_n(0)$-module corresponding to the set $\{s_i:i\in I\}$ of simple generators of $S_n$. Similarly, for $I\subseteq [0,n-1]$, we write $\mathcal S_I^B$ for the simple $H_n^B(0)$-module corresponding to the set $\{s_i:i\in I\}$ of simple generators of $B_n$. The module structures of $\mathcal S_I$ and $\mathcal S_I^B$ are determined by 
\[\pi_j\varepsilon_I = \begin{cases} -\varepsilon_I & \mbox{ if } j\in I \\
                                       0 & \mbox{ if } j\notin I,
                                       \end{cases}\]
where $\varepsilon_I$ is a fixed basis element of $\mathcal S_I$ or $\mathcal S_I^B$.                

Let $[M]$ denote the isomorphism class of a $0$-Hecke module $M$. The \dfn{Grothendieck group} of $H_n(0)$, denoted $\mathcal{G}_0(H_n(0))$, is the span of the isomorphism classes of the finitely-generated $H_n(0)$-modules, modulo the relation $[A]-[B]+[C]=0$ whenever there exists a short exact sequence $0\rightarrow A \rightarrow B \rightarrow C\rightarrow 0$; the Grothendieck group $\mathcal{G}_0(H_n^B(0))$ is defined analogously. The (type-$A$) \dfn{quasisymmetric characteristic} $\ch: \bigoplus_{n\ge 0} \mathcal{G}_0(H_n(0)) \rightarrow \QSym$ \cite{DKLT} is the isomorphism defined by $\ch([\mathcal S_I]) = F_I$. The \dfn{type-$B$ quasisymmetric characteristic} $\ch^B:\mathcal{G}_0(H_n^B(0))\rightarrow \QSym^B$ is defined by $\ch^B([\mathcal S_I^B]) = F_I^B$. For ease of notation, we will generally write $\ch(M)$ in place of $\ch([M])$.

\section{Ascent-Compatibility}\label{sec:ascentcompatible}

Let $(W,S)$ be a Coxeter system. A \dfn{reduced word} for an element $w\in W$ is a minimum-length word over the alphabet $S$ that represents $w$; the number of letters in a reduced word for $w$ is called the \dfn{length} of $w$ and is denoted by $\ell(w)$. We say $s\in S$ is a (left) \dfn{descent} of $w$ if $\ell(sw)<\ell(w)$; otherwise, $s$ is an \dfn{ascent} of $w$. The set of descents of $w$ is denoted $\Des(w)$. 

Given an arbitrary set $X$, we write $\mathbb CX$ for the complex vector space freely generated by $X$.

Fix $X\subseteq W$. For $s\in S$, define a linear operator $ \pi_s: \mathbb C X \rightarrow \mathbb C X$ by
\[ \pi_s(x) = \begin{cases} -x & \mbox{ if } s\in \Des(x) \\
				       0 & \mbox{ if } s\notin \Des(x) \mbox{ and } sx\notin X \\
				       sx & \mbox{ if } s\notin \Des(x) \mbox{ and } sx\in X
\end{cases}\]
for all $x\in X$. For $u,v\in W$ and $s,t\in S$, let us say the quadruple $(u,v,s,t)$ is \dfn{aligned} if the simple generator $s$ is an ascent of $u$, the simple generator $t$ is an ascent of $v$, and $u^{-1}su=v^{-1}tv$. We say $X$ is \dfn{ascent-compatible} if for all $u,v \in X$ and all $s,t\in S$ such that $(u,v,s,t)$ is aligned, we have that $su\in X$ if and only if $tv\in X$. 

The following theorem was proven in type~$A$ in \cite{Searles:0HC}, where it was used to provide a general framework for constructing type-$A$ 0-Hecke modules whose quasisymmetric characteristics are interesting quasisymmetric functions. 

\begin{theorem}\label{thm:ascentcompatible}
Let $(W,S)$ be a Coxeter system. If $X$ is an ascent-compatible subset of $W$, then the linear operators $\pi_s$ define an action of the $0$-Hecke algebra $H_W(0)$ on $\mathbb C X$.
\end{theorem}

\begin{proof}
Fix $x\in X$ and $s\in S$. We first show that $ \pi_s^2(x)=- \pi_s(x)$. If $ \pi_s(x)=-x$, then we have $ \pi_s^2(x)= \pi_s(-x)=- \pi_s(x)$. If $ \pi_s(x)=0$, then $ \pi_s^2(x)=0=- \pi_s(x)$. If $ \pi_s^2(x)=sx$, then $s\not\in\Des(x)$, so $s\in\Des(sx)$. In this case, $ \pi_s^2(x)= \pi_s(sx)=-sx=- \pi_s(x)$. 

Now choose some $t\in S\setminus\{s\}$, and let $m=m(s,t)$. We need to show that ${[{\pi}_s\vert{\pi}_t]_m(x)=[{\pi}_t\vert{\pi}_s]_m(x)}$. We consider two cases. 

\medskip 

\noindent {\bf Case 1.} Suppose either $s\in\Des(x)$ or $t\in\Des(x)$. Without loss of generality, we may assume $s\in\Des(x)$. Then ${\pi}_s(x)=-x$, so $[{\pi}_s\vert{\pi}_t]_i(x)=-[{\pi}_t\vert{\pi}_s]_{i+1}(x)$ for all $i\geq 0$. Let $k$ be the largest integer in $[0,m-1]$ such that $[{\pi}_s\vert{\pi}_t]_k(x)=[s\vert t]_kx$. Let 
\[a=\begin{cases} s & \mbox{if } k\text{ is even} \\ t & \mbox{if } k\text{ is odd}\end{cases}\quad\text{and}\quad \overline a=\begin{cases} t & \mbox{if } k\text{ is even} \\ s & \mbox{if } k\text{ is odd,}\end{cases}\] and observe that $a\in\Des([s\vert t]_kx)$. If we assume that $k\leq m-2$ and $[{\pi}_s\vert{\pi}_t]_{k+1}(x)=0$, then ${[{\pi}_t\vert{\pi}_s]_{k+2}(x)=-[{\pi}_s\vert{\pi}_t]_{k+1}(x)=0}$, so ${[{\pi}_s\vert{\pi}_t]_m(x)=0=[{\pi}_t\vert{\pi}_s]_m(x)}$. Hence, we may assume either that $k=m-1$ or that $k\leq m-2$ and $[{\pi}_s\vert{\pi}_t]_{k+1}(x)=-[s\vert t]_kx$. We claim that $\overline a\in\Des([s\vert t]_kx)$. This is immediate from the definitions of $k$ and ${\pi}_{\overline a}$ if $[{\pi}_s\vert{\pi}_t]_{k+1}(x)=-[s\vert t]_kx$. Now suppose $k=m-1$. Since $s\in\Des(x)$, we can write $x=sy$, where $\ell(x)=\ell(y)+1$. Then $[s\vert t]_{m-1}x=[t\vert s]_my=[s\vert t]_my$, and $\ell([s\vert t]_my)=\ell(y)+m$. The first letter in the word $[s\vert t]_my$ is $\overline a$, so $\overline a\in\Des([s\vert t]_my)=\Des([s\vert t]_{k}x)$.

We have proven that $\{s,t\}=\{a,\overline a\}\subseteq\Des([s\vert t]_kx)$, so $[{\pi}_s\vert{\pi}_t]_m(x)=(-1)^{m-k}[s\vert t]_k(x)$ and $[{\pi}_t\vert{\pi}_s]_m(x)=-[{\pi}_s\vert{\pi}_t]_{m-1}(x)=-(-1)^{m-1-k}[s\vert t]_k(x)$. Thus, $[{\pi}_s\vert {\pi}_t]_m(x)=[{\pi}_t\vert {\pi}_s]_m(x)$. 

\medskip 

\noindent {\bf Case 2.} Suppose $s,t\not\in\Des(x)$. Since $[s\vert t]_m$ is the longest element of the parabolic subgroup of $W$ generated by $\{s,t\}$, it follows from \cite[Lemma~3.2.4]{BjornerBrenti} that $\ell([s\vert t]_mx)=\ell(x)+m$. This implies that $[{\pi}_s\vert{\pi}_t]_m(x)\in\{0,[s\vert t]_mx\}$ and $[{\pi}_t,{\pi}_s]_m(x)\in\{0,[t\vert s]_mx\}=\{0,[t\vert s]_mx\}$. Assume by way of contradiction that $[{\pi}_s,\vert{\pi}_t]_m(x)\neq[{\pi}_t\vert{\pi}_s]_m(x)$. Without loss of generality, we may assume that $[{\pi}_s\vert{\pi}_t]_m(x)=0$ and $[{\pi}_t\vert{\pi}_s]_m(x)=[s\vert t]_mx$. There exists $j\in[0,m-1]$ such that ${[{\pi}_s\vert{\pi}_t]_j(x)\neq 0=[{\pi}_s\vert{\pi}_t]_{j+1}(x)}$. Let 
\[b=\begin{cases} t & \mbox{if } j\text{ is even} \\ s & \mbox{if } j\text{ is odd}\end{cases}\quad\text{and}\quad c=\begin{cases} s & \mbox{if } m-1-j\text{ is even} \\ t & \mbox{if } m-1-j\text{ is odd.}\end{cases}\] Then \[([s\vert t]_jx)^{-1}b([s\vert t]_jx)=x^{-1}[s\vert t]_{2j+1}x=x^{-1}[t\vert s]_{2m-2j-1}x=([t\vert s]_{m-1-j}x)^{-1}c([t\vert s]_{m-1-j}x),\] and we have \[[s\vert t]_jx\in X,\quad [t\vert s]_{m-1-j}x\in X,\quad b\in S\setminus\Des([s\vert t]_jx),\quad \text{and}\quad c\in S\setminus\Des([t\vert s]_{m-1-j}x).\] This shows that the quadruple $([s\vert t]_jx,[t\vert s]_{m-1-j}x,b,c)$ is aligned. Furthermore, we know that ${\pi}_b([s\vert t]_jx)=[{\pi}_s\vert{\pi}_t]_{j+1}(x)=0$, so $b[s\vert t]_jx\not\in X$. Because $X$ is ascent-compatible, we must have ${c[t\vert s]_{m-1-j}x\not\in X}$, so ${\pi}_c([t\vert s]_{m-1-j}x)=0$. This is impossible since ${\pi}_c([t\vert s]_{m-1-j}x)=[{\pi}_t\vert{\pi}_s]_{m-j}(x)$ and $[{\pi}_t\vert{\pi}_s]_{m}(x)\neq 0$. 
\end{proof}

For an easy example of an ascent-compatible set, consider $I\subseteq S$, and define the \dfn{left descent class} associated to $I$ to be the set $D_I^W=\{x\in W:\Des(x)=I\}$. 

\begin{proposition}\label{prop:leftdescent}
Let $(W,S)$ be a Coxeter system. If $I\subseteq S$ and $X\subseteq D_I^W$, then $X$ is ascent-compatible. 
\end{proposition}

\begin{proof}
Suppose $u,v\in X$ and $s,t\in S$ are such that $(u,v,s,t)$ is aligned. Then $s\not\in\Des(u)=I$ and $t\not\in\Des(v)=I$. However, we have $s\in\Des(su)$ and $t\in\Des(tv)$, so $su$ and $tv$ are not in $D_I^W$. It follows that $su$ and $tv$ are both not in $X$. 
\end{proof}

The next theorem provides more interesting examples of ascent-compatible sets. First, we need to establish some additional background. As before, let $(W,S)$ be a Coxeter system. For $x,y\in W$, let us write $x\leq_{\LL}y$ if there is a reduced word for $y$ that contains a reduced word for $x$ as a suffix; then $\leq_{\LL}$ is a partial order on $W$ called the \dfn{left weak order}. If $x\leq_{\LL}y$, then the \dfn{interval} between $x$ and $y$ in the left weak order is the set $[x,y]_{\LL}=\{z\in W:x\leq_{L}z\leq_{\LL}y\}$. A set $X\subseteq W$ is \dfn{convex} (in the left weak order) if for all $x,y\in X$ with $x\leq_{\LL}y$, we have $[x,y]_{\LL}\subseteq X$. A \dfn{reflection} in $W$ is an element that is conjugate to a simple generator. Thus, the set of reflections is $\{w^{-1}sw:s\in S, w\in W\}$. A reflection $r$ is called a \dfn{right inversion} of an element $x\in W$ if $\ell(xr)<\ell(x)$. Let $\Inv(x)$ denote the set of right inversions of $x$. For $x,y\in W$, it is well known that $\ell(x)=\lvert\Inv(x)\rvert$ and that $x\leq_{\LL}y$ if and only if $\Inv(x)\subseteq\Inv(y)$. 

\begin{theorem}\label{thm:convex}
Let $(W,S)$ be a Coxeter system, and let $X\subseteq W$. If $X$ is convex in the left weak order and has a unique maximal element under the left weak order, then $X$ is ascent-compatible.   
\end{theorem}

\begin{proof}
Let $z$ be the unique maximal element of $X$. Assume by way of contradiction that $X$ is not ascent-compatible. Then there exist $u,v\in X$ and $s,t\in S$ such that $(u,v,s,t)$ is aligned, $su\in X$, and $tv\not\in X$. Let $r=u^{-1}su=v^{-1}tv$. Since $u\leq_{\LL} su$ and $v\leq_{\LL}tv$, the reflection $r$ is both a right inversion of $su$ and a right inversion of $tv$. Since $su\in X$, we have $su\leq_{\LL}z$, so $\Inv(su)\subseteq\Inv(z)$. This implies that $r\in\Inv(z)$. On the other hand, since $\lvert\Inv(tv)\rvert=\ell(tv)=\ell(v)+1=\lvert\Inv(v)\rvert+1$ and $r\not\in\Inv(v)$, we must have $\Inv(tv)=\Inv(v)\cup\{r\}$. Because $v\in X$, we have $v\leq_{\LL}z$, so $\Inv(v)\subseteq\Inv(z)$. This implies that $\Inv(tv)=\Inv(v)\cup\{r\}\subseteq\Inv(z)$, so $v\leq_{\LL}tv\leq_{\LL}z$. However, $X$ is convex and contains $v$ and $z$, so this forces $tv$ to belong to $X$, which is a contradiction.      
\end{proof}

\begin{remark}\label{rem:inverse_descent}
Inverses of left descent classes (which are also called \emph{right descent classes}) provide notable examples of convex subsets of the left weak order of a Coxeter group $W$. More precisely, if $(W,S)$ is a Coxeter system and $I\subseteq S$ is such that the left descent class $D_I^W$ is nonempty, then the right descent class $(D_I^W)^{-1}=\{x\in W:\Des(x^{-1})=I\}$ is a finite convex subset of the left weak order that has a unique maximal element. Hence, if $D_I^W$ is nonempty, then $(D_I^W)^{-1}$ is ascent-compatible by \cref{thm:convex}. 
\end{remark}

It was shown in \cite{JKLO} that intervals in the left weak order on the symmetric group $S_n$ give rise to important families of type-$A$ $0$-Hecke modules, and in \cite{Searles:0HC}, it was shown that these modules can be obtained via the ascent-compatibility framework.  
One can view \cref{thm:convex} as a generalization of this to arbitrary Coxeter groups; our theorem is also more general even in type~$A$ since it does not require the set $X$ to have a unique minimal element. 

Let us now specialize our attention to the hyperoctahedral group $B_n$ and its $0$-Hecke algebra $H_n^B(0)$. 
It will be convenient to identify simple generators of $B_n$ with their indices; for instance, we will view the descent set of an element of $B_n$ as a subset of $[0,n-1]$. 
If $X$ is an ascent-compatible subset of $S_n$, then \Cref{thm:ascentcompatible} endows $\mathbb C X$ with the structure of an $H_n(0)$-module. It was proven in \cite{Searles:0HC} that the quasisymmetric characteristic of this module is given by \[\ch(\mathbb C X)=\sum_{x\in X}F_{\Des(x)}.\] The same argument yields the following analogue of this result in type~$B$. 

\begin{theorem}\label{thm:ascent_compatible_B}
If $X$ is an ascent-compatible subset of $B_n$, then \[\ch^B(\mathbb C X)=\sum_{x\in X}F^B_{\Des(x)}.\]
\end{theorem}

\begin{proof}
Let $m=|X|$, and let $x_m,\ldots,x_1$ be a linear extension of the poset $(X,\leq_{\LL})$. That is, $x_m,\ldots,x_1$ is an ordering of the elements of $X$ such that $j<i$ whenever $x_i\leq_{\LL}x_j$. For $0\leq k\leq m$, let ${\bf N}_k=\text{span}\{x_1,\ldots,x_k\}$. Then \[0={\bf N}_0\subseteq {\bf N}_1\subseteq\cdots\subseteq{\bf N}_m=\mathbb C X\] is a filtration of $\mathbb C X$. For $1\leq k\leq m$, the quotient ${\bf N}_k/{\bf N}_{k-1}$ is a $1$-dimensional $H_n^B(0)$-module spanned by the element $\overline x_k=x_k+{\bf N}_{k-1}$; the module structure is given by \[\pi_j(\overline x_k)=\begin{cases} -\overline x_k & \mbox{ if } j\in \Des(x_k) \\
0 & \mbox{ if } j\not\in\Des(x_k).
\end{cases}\] Thus, ${\bf N}_k/{\bf N}_{k-1}$ is isomorphic to the simple module $\mathcal S_{\Des(x_k)}^B$. It follows that \[\ch^B(\mathbb C X)=\sum_{k=1}^m\ch^B({\bf N}_k/{\bf N}_{k-1})=\sum_{k=1}^m\ch^B(\mathcal S_{\Des(x_k)}^B)=\sum_{k=1}^m F^B_{\Des(x_k)}. \qedhere \]
\end{proof}

In \cref{sec:CliffordApplications}, we will obtain a result similar to \cref{thm:ascent_compatible_B} that will allow us to compute the type-$B$ quasisymmetric characteristic of the module obtained from an ascent-compatible set $X$ by first inducing a type-$B$ $0$-Hecke--Clifford module from $\mathbb CX$ and then restricting back to an $H_n^B(0)$-module. 

We now provide specific applications of \Cref{thm:ascent_compatible_B}. Motivated by a type-$B$ analogue of Schur positivity, Mayorova and Vassilieva \cite{MV} studied several interesting subsets of $B_n$ and functions associated to them. For $X\subseteq B_n$, let \[\mathcal Q(X)=\sum_{x\in X}F_{\Des(x^{-1})}^B.\] For a specific choice of $X$, if we can show that the set $X^{-1} = \{x^{-1}:x\in X\}$ is ascent-compatible, then it will follow from \Cref{thm:ascentcompatible,thm:ascent_compatible_B} that $\mathcal Q(X)$ can be realized as the quasisymmetric characteristic of an $H_n^B(0)$-module. 

We view $B_n$ as the group of permutations $\sigma$ of the set $\pm[n]:=\{\pm1,\ldots,\pm n\}$ with the property that ${\sigma(-i)=-\sigma(i)}$ for all $i$. One example of a set $X\subseteq B_n$ whose corresponding function $\mathcal{Q}(X)$ was studied in \cite{MV} is the left descent class $D_I^{B} = \{\sigma\in B_n: \Des(\sigma)=I\}$ associated to a set $I\subseteq [0,n-1]$ (where we identify the simple generator $s_i$ with the index $i$). A second example of such a set $X$ is the set $L^B=\bigcup_{i=1}^nL_i^B$, where $L_i^B$ is the set of all $\sigma\in B_n$ such that \[\sigma^{-1}(1) > \cdots > \sigma^{-1}(i) < \cdots < \sigma^{-1}(n);\] elements of $L^B$ are called \dfn{left-unimodal permutations}.

\begin{theorem}
The sets $(D_I^{B})^{-1}$ and $(L_i^B)^{-1}$ are ascent-compatible. Therefore, both $\mathbb C(D_I^{B})^{-1}$ and $\mathbb C(L_i^B)^{-1}$ are $H_n^B(0)$-modules, and we have 
\[\mathrm{ch}^B(\mathbb C(D_I^{B})^{-1}) = \mathcal{Q}(D_I^{B})  \,\,\,\, \mbox{ and } \,\,\,\, \mathrm{ch}^B(\mathbb C(L_i^B)^{-1})=\mathcal{Q}(L_i^B);\]
the latter of these identities implies that $\mathrm{ch}^B(\bigoplus_{i=1}^n\mathbb C(L_i^B)^{-1})=\mathcal{Q}(L^B)$.
\end{theorem}

\begin{proof}
We know from \cref{rem:inverse_descent} that $(D_I^B)^{-1}$ is ascent-compatible. 

Fix $1\leq i\leq n$. Let $\sigma$ be the element of $B_n$ such that $\sigma(j)=-j$ for $1\leq j\leq i-1$ and $\sigma(k)=k-i-n$ for $i\leq k\leq n$. Let $\tau$ be the element of $B_n$ such that $\tau(j)=i-j$ for $1\leq j\leq i-1$ and $\tau(k)=k$ for $i\leq k\leq n$. It is straightforward to check that $(L_i^B)^{-1}$ is equal to the interval $[\sigma,\tau]_{\LL}$ in the left weak order. According to \cref{thm:convex}, $(L_i^B)^{-1}$ is ascent-compatible. 
\end{proof}

We note that $D_I^B$ and $L_i^B$ themselves are subsets of left descent classes, so they are ascent-compatible by \cref{prop:leftdescent}. Therefore, the functions $\mathcal{Q}(D_I^{B})^{-1}$ and $\mathcal{Q}(L_i^B)^{-1}$ are also quasisymmetric characteristics of $H_n^B(0)$-modules by \Cref{thm:ascentcompatible,thm:ascent_compatible_B}.

Another subset of $B_n$ considered in \cite{MV} is the \dfn{type-$B$ Knuth class} associated to a standard domino tableau $T$. This is the set $C_T^B$ of all $\sigma\in B_n$ that insert to $T$ under a certain type-$B$ analogue of the Robinson--Schensted correspondence (see \Cref{sec:domino} for the definition of a standard domino tableau). The set $(C_T^B)^{-1}$ is not in general ascent-compatible. However, $C_T^B$ itself is a subset of a left-descent class, so it is ascent-compatible, and the functions $\mathcal{Q}((C_T^B)^{-1})$ are quasisymmetric characteristics of $H_n^B(0)$-modules. 

The final example considered in \cite{MV} concerns the \emph{signed arc permutations} in $B_n$. The set of signed arc permutations is ascent-compatible (although the set of their inverses is not); this ascent-compatibility is noteworthy because the set of signed arc permutations in $B_n$ is \emph{not} convex in the left weak order, so its ascent-compatibility does not follow from \cref{thm:convex}.

Let $-[n]=\{-1,\ldots,-n\}$. It will be convenient to represent negative numbers using overlines; for example, we will write $\widebar 3$ instead of $-3$. The \dfn{one-line notation} of a signed permutation $\sigma\in B_n$ is the word $\sigma(1)\cdots\sigma(n)$. Suppose $x_0$ and $x_1$ are finite words that use disjoint sets of letters. A \dfn{shuffle} of $x_0$ and $x_1$ is a word $w$ such that for each $i\in\{0,1\}$, deleting the letters in $w$ that appear in $x_{1-i}$ yields the word $x_i$. Let $x_0\shuffle x_1$ denote the set of all shuffles of $x_0$ and $x_1$. For example, the elements of $24\shuffle\widebar 3\widebar 1$ are \[24\widebar 3\widebar 1,\quad 2\widebar 34\widebar 1,\quad \widebar 324\widebar 1,\quad 2\widebar3\widebar 14,\quad \widebar 32\widebar 14,\quad \widebar 3\widebar 124.\] 

A \dfn{signed arc permutation} is an element of $B_n$ whose one-line notation is a shuffle of a word $a_1\cdots a_p$ over the alphabet $[n]$ and a word $b_1\cdots b_{n-p}$ over the alphabet $-[n]$ satisfying the following conditions: 
\begin{itemize}
\item $a_{i+1}\equiv a_i+1\pmod n$ for all $1\leq i\leq p-1$;
\item $b_{j+1}\equiv b_j+1\pmod n$ for all $1\leq j\leq  n-p-1$; 
\item $a_1\equiv -b_1+1\pmod n$ if $p$ and $n-p$ are both nonzero. 
\end{itemize}
Let $\mathcal A_n^B$ be the set of signed arc permutations in $B_n$. For example, \[\mathcal A_3^B=\{123,231,312,\widebar 3\widebar 2\widebar 1,\widebar 2\widebar 1\widebar 3,\widebar 1\widebar 3\widebar 2\}\cup(12\shuffle\widebar 3)\cup(31\shuffle \widebar 2)\cup(23\shuffle \widebar 1)\cup(1\shuffle \widebar 3\widebar 2)\cup(2\shuffle \widebar 1\widebar 3)\cup(3\shuffle\widebar 2\widebar 1).\]

\begin{theorem}
The set $\mathcal A_n^B$ is ascent-compatible. Therefore, $\mathbb C\mathcal A_n^B$ is an $H_n^B(0)$-module, and \[\ch^B(\mathbb C\mathcal A_n^B)=\mathcal Q((\mathcal A_n^B)^{-1}).\] 
\end{theorem}

\begin{proof}
The proof is easy if $n\leq 2$, so assume $n\geq 3$. Let $u$ and $v$ be signed arc permutations in $B_n$, and suppose $s,t\in S$ are such that $(u,v,s,t)$ is aligned. We must show that $su\in\mathcal A_n^B$ if and only if $tv\in\mathcal A_n^B$. 

None of the simple generators of $B_n$ other than $s_0$ are conjugate to $s_0$. Since $u^{-1}su=v^{-1}tv$, we have $s=s_0$ if and only if $t=s_0$. If $s=t=s_0$, then the reflection $u^{-1}su=v^{-1}tv$ is a transposition of the form $(k\,\,\widebar k)$ for some $k\in[n]$. On the other hand, if $s$ and $t$ are not $s_0$, then the reflection $u^{-1}su=v^{-1}tv$ can be written as $(m_1\,\,m_2)(\widebar m_1\,\,\widebar m_2)$ for some $m_1,m_2\in\pm[n]$. We now consider four cases. 

\medskip 
\noindent {\bf Case 1.} Suppose that $s=s_0$ and that the reflection $u^{-1}su=v^{-1}tv$ is either $(1\,\,\widebar 1)$ or $(n\,\,\widebar n)$. We consider only the case when this reflection is $(1\,\,\widebar 1)$; the other case is similar. Because $s_0$ is an ascent of $u$, we have $u(1)=\widebar 1$ and $(su)(1)=1$. There exists $q\in[n]$ such that $u\in(23\cdots(q-1))\shuffle(\widebar 1\widebar n\cdots\widebar q)$ and $su\in(12\cdots(q-1))\shuffle(\widebar n\cdots\widebar q)$. It follows that $su=s_0u\in\mathcal A_n^B$. A similar argument shows that $tv=s_0v\in\mathcal A_n^B$ as well. 

\medskip 
\noindent {\bf Case 2.} Suppose that $s=s_0$ and that $u^{-1}su=v^{-1}tv=(k\,\,\widebar k)$ for some $2\leq k\leq n-1$. We will show that $su,tv\not\in\mathcal A_n^B$. Let us prove that $su\not\in\mathcal A_n^B$; the same argument can be used to prove that $tv\not\in\mathcal A_n^B$. Let us write $u$ as a shuffle of a word $a_1\cdots a_p$ over $[n]$ and a word $b_1\cdots b_{n-p}$ over $-[n]$ as in the definition of signed arc permutations. Because $s_0$ is an ascent of $u$, we have $b_j=\widebar 1$ for some $1\leq j\leq n-p$. Suppose by way of contradiction that $su\in\mathcal A_n^B$. The one-line notation of $su$ is obtained from that of $u$ by changing $\widebar 1$ into $1$. We claim that negative entries cannot appear in the one-line notation of $u$ (equivalently, $su$) on both sides of position $k$; indeed, if there were negative entries on both sides of position $k$, then deleting the positive entries from the one-line notation of $su$ would result in a word over $-[n]$ in which $\widebar 2$ and $\widebar n$ appear consecutively, contradicting the definition of $\mathcal A_n^B$. Similarly, positive entries cannot appear in the one-line notation of $u$ (equivalently, $su$) on both sides of position $k$. It follows that $j$ is $1$ or $n-p$. We will assume that $j=1$; the argument is similar if $j=n-p$. Thus, the one-line notation of $u$ is $a_1\cdots a_p\widebar 1b_2\cdots b_{n-p}$, while the one-line notation of $su$ is $a_1\cdots a_p1b_2\cdots b_{n-p}$. The defining conditions of signed arc permutations then imply that $a_p=n$ and $b_2=\widebar n$, which is impossible.  

\medskip 
\noindent {\bf Case 3.} Suppose that $s$ and $t$ are not $s_0$ and that $u^{-1}su=v^{-1}tv=(m_1\,\,m_2)(\widebar m_1\,\,\widebar m_2)$, where the set $\{m_1,m_2,\widebar m_1,\widebar m_2\}$ is not equal to $\{n-1,n,\widebar{n-1},\widebar n\}$. Without loss of generality, we may assume that $m_1$ is the smallest positive element of $\{m_1,m_2,\widebar m_1,\widebar m_2\}$. Then $1\leq m_1\leq n-2$. In this case, we will show that $su$ and $tv$ are not in $\mathcal A_n^B$. Let us prove that $su\not\in\mathcal A_n^B$; the same argument can be used to prove that $tv\not\in\mathcal A_n^B$. Let $s=s_i$. Let us say a $q$-element set $J\subseteq\mathbb Z$ is a \dfn{cyclic interval} modulo $n$ if there is a set of $q$ consecutive integers that are congruent modulo $n$ to the elements of $J$. Observe that if $\sigma\in\mathcal A_n^B$, then for each $q\in[n]$, the set $\{|\sigma(j)|:1\leq j\leq q\}$ is a cyclic interval modulo $n$. Now, \[\{|(su)(j)|:1\leq j\leq m_1\}=(\{|u(j)|:1\leq j\leq m_1\}\setminus\{i\})\cup\{i+1\}.\] Since $m_1\leq n-2$, this implies that the sets $\{|u(j)|:1\leq j\leq m_1\}$ and $\{|(su)(j)|:1\leq j\leq m_1\}$ cannot both be cyclic intervals modulo $n$. Since $u\in\mathcal A_n^B$, we must have $su\not\in\mathcal A_n^B$. 

\medskip 
\noindent {\bf Case 4.} Suppose that $s$ and $t$ are not $s_0$ and that $u^{-1}su=v^{-1}tv=(m_1\,\,m_2)(\widebar m_1\,\,\widebar m_2)$, where $\{m_1,m_2,\widebar m_1,\widebar m_2\}=\{n-1,n,\widebar{n-1},\widebar n\}$. We will show that $su$ and $tv$ are not in $\mathcal A_n^B$. Let us prove that $su\not\in\mathcal A_n^B$; the same argument can be used to prove that $tv\not\in\mathcal A_n^B$. Let us write $u$ as a shuffle of a word $a_1\cdots a_p$ over $[n]$ and a word $b_1\cdots b_{n-p}$ over $-[n]$ as in the definition of signed arc permutations. Let $s=s_i$. Because $s_i$ is an ascent of $u$, the pair $(u(n-1),u(n))$ must be $(i,i+1)$, $(\widebar{i+1},\widebar i)$, $(\widebar i,i+1)$, or $(i+1,\widebar i)$. Suppose first that this pair is $(\widebar i,i+1)$ or $(i+1,\widebar i)$. Then $a_p=i+1$ and $b_{n-p}=\widebar i$. If $p\geq 2$, then $a_{p-1}=a_p-1=i$; if $p=1$, then $b_{n-2}=b_{n-1}-1=\widebar{i+1}$. In either case, we obtain a contradiction because $i$ and $\widebar{i+1}$ do not appear in the one-line notation of $u$. Thus, we may assume that the pair $(u(n-1),u(n))$ is either $(i,i+1)$ or $(\widebar{i+1},\widebar i)$. If $(u(n-1),u(n))=(i,i+1)$, then the last two entries in the one-line notation of $su$ are $i+1,i$, so it follows from the first bulleted item in the definition of a signed arc permutation that $su\not\in\mathcal A_n^B$. Similarly, if $(u(n-1),u(n))=(\widebar{i+1},\widebar i)$, then the last two entries in the one-line notation of $su$ are $\widebar i,\widebar{i+1}$, so it follows from the second bulleted item that $su\not\in\mathcal A_n^B$.  
\end{proof}

\section{Domino Tableaux and Type-$B$ Peak Functions}\label{sec:domino}

A central motivation for studying the families of functions considered in \Cref{sec:ascentcompatible} is the fact that they expand positively into type-$B$ analogues of Schur functions known as \emph{domino functions}, which have formulas in terms of domino tableaux \cite{MV}. In this section, we define an action of the type-$B$ $0$-Hecke algebra on standard domino tableaux, thereby obtaining a representation-theoretic interpretation of the domino functions. We then consider the shifted domino tableaux introduced by Chemli \cite{Chemli}, define a modification of these that allows $0$ entries, and introduce standard shifted domino tableaux. We show that the generating functions of the (modified) shifted domino tableaux, which we call \emph{shifted domino functions}, expand positively in the type-$B$ peak functions, with the expansion indexed by standard shifted domino tableaux. Similarly to how domino functions form a type-$B$ analogue of Schur functions, the shifted domino functions may be regarded as forming a type-$B$ analogue of Schur $Q$-functions. 

\subsection{Standard domino tableaux}
Let $\lambda\vdash 2n$. A \dfn{domino tiling} of $\lambda$ is a tiling of the Young diagram of shape $\lambda$ by $1\times 2$ and $2\times 1$ rectangles called \dfn{dominoes}. A \dfn{standard domino tableau} of shape $\lambda$ is a bijective filling of a domino tiling of $\lambda$ with entries from $[n]$ such that entries increase from left to right along rows and from top to bottom in columns (see \Cref{Fig1}). Let $\SDT(\lambda)$ denote the set of standard domino tableaux of shape $\lambda$.

\begin{figure}[ht]
  \begin{center}{\includegraphics[height=2cm]{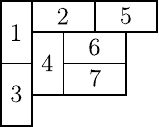}}
  \end{center}
  \caption{A standard domino tableau of shape $(5,4,4,1)\vdash 14$.}\label{Fig1}
\end{figure}

Suppose $T\in\SDT(\lambda)$. Let $\dom_i(T)$ denote the domino in $T$ filled with the entry $i$. A number $i\in [0,n-1]$ is a \dfn{descent} of $T$ if $i=0$ and $\dom_1(T)$ is vertical or if $i>0$ and $\dom_{i+1}(T)$ is strictly lower than $\dom_i(T)$. Let $\Des(T)$ be the set of descents of $T$. For example, the descent set of the domino tableau in \Cref{Fig1} is $\{0,2,5,6\}$. The \dfn{domino function} $G_\lambda$ \cite{MV} is defined by 
\[G_\lambda = \sum_{T\in \SDT(\lambda)}F^B_{\Des(T)}.\]
For $1\le i \le n-1$, let $s_i(T)$ be the tableau obtained by exchanging the entries $i$ and $i+1$. If $\dom_1(T)$ and $\dom_2(T)$ tile the northwest-most $2\times 2$ square in $T$, then we let $s_0(T)$ be the tableau obtained by switching the orientation of both these dominoes (i.e., from horizontal to vertical or vice versa); otherwise, we say $s_0(T)$ is not defined (in particular, we would say $s_0(T)\not\in\SDT(\lambda)$ in this case). If $s_0(T)$ is defined and $1$ appears above (respectively, left of) $2$ in $T$, then we place $1$ left of (respectively, above) $2$ in $s_0(T)$. We define operators $\pi_0, \pi_1, \ldots \pi_{n-1}$ on $\mathbb C\SDT(\lambda)$ (the complex vector space with basis $\SDT(\lambda)$) by
\[\pi_i(T) = \begin{cases} -T & \mbox{ if } i\in\Des(T) \\
                            0 & \mbox{ if } i\notin\Des(T)\mbox{ and }s_i(T)\notin\SDT(\lambda) \\
                            s_i(T) & \mbox{ if } i\notin\Des(T)\mbox{ and }s_i(T)\in\SDT(\lambda).
\end{cases}\]

\begin{example}
Let $\lambda=(4,3,3)$ and 
\[T=\begin{array}{l}
\includegraphics[height=1.507cm]{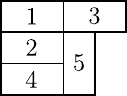}
\end{array}\in\SDT(\lambda).\] Then we have \[\pi_0(T)=\begin{array}{l}
\includegraphics[height=1.507cm]{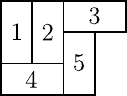}\end{array} \qquad\text{and}\qquad \pi_2(T)=\begin{array}{l} \includegraphics[height=1.507cm]{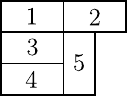}
\end{array}.\]
We also have $\pi_1(T)=\pi_3(T)=-T$ and $\pi_4(T)=0$. 
\end{example}

\begin{theorem}\label{thm:dominoHecke}
Let $\lambda\vdash 2n$. The operators $\pi_0, \pi_1, \ldots , \pi_{n-1}$ define an action of $H_n^B(0)$ on $\mathbb C\SDT(\lambda)$. 
\end{theorem}
\begin{proof}
We need to demonstrate that $\pi_0, \pi_1, \ldots , \pi_{n-1}$ satisfy the type-$B$ $0$-Hecke relations. Fix $T\in\SDT(\lambda)$. First, we show that $\pi_i^2(T)=-\pi_i(T)$. This is immediate if $\pi_i(T)=0$ or $\pi_i(T)=-T$. If $\pi_i(T)=s_i(T)$, then $i$ is a descent of $s_i(T)$, so $\pi_i^2(T) = -s_i(T)=-\pi_i(T)$.

Next, we show that $\pi_i\pi_j(T) = \pi_j\pi_i(T)$ whenever $j-i\geq 2$. If $i,j\geq 1$, then this is immediate since the action of $\pi_i$ only affects the dominoes with entries $i$ and $i+1$ while the action of $\pi_j$ only affects the dominoes with entries $j$ and $j+1$. A similar argument handles the case when $i=0$ and $j\geq 3$ since the action of $\pi_0$ can only affect the dominoes with entries $1$ and $2$. Now suppose $i=0$ and $j=2$; we need to show that $\pi_0\pi_2(T) = \pi_2\pi_0(T)$. If $\dom_1(T)$ is vertical, then $\pi_0(T)=-T$, so we need to check that $\pi_0\pi_2(T)=-\pi_2(T)$. This is immediate if $\pi_2(T)=0$. If $\pi_2(T)\neq 0$, then there is a unique tableau $T'$ in the support of $\pi_2(T)$, and $\dom_1(T')$ is vertical. In this case, $\pi_0(T')=-T'$, so the desired identity follows. Now assume $\dom_1(T)$ is horizontal. If there is not a horizontal domino immediately below $\dom_1(T)$ in $T$, then it is straightforward to check that $\pi_0\pi_2(T)$ and $\pi_2\pi_0(T)$ are both $0$. Thus, we may assume that the northwest-most $2\times 2$ square in $T$ is tiled by two horizontal dominoes. There are six remaining cases that depend on the relative positions of the entries $1,2,3$ in $T$; we can check them by hand as follows: 
\[\begin{array}{l}\includegraphics[height=10.195cm]{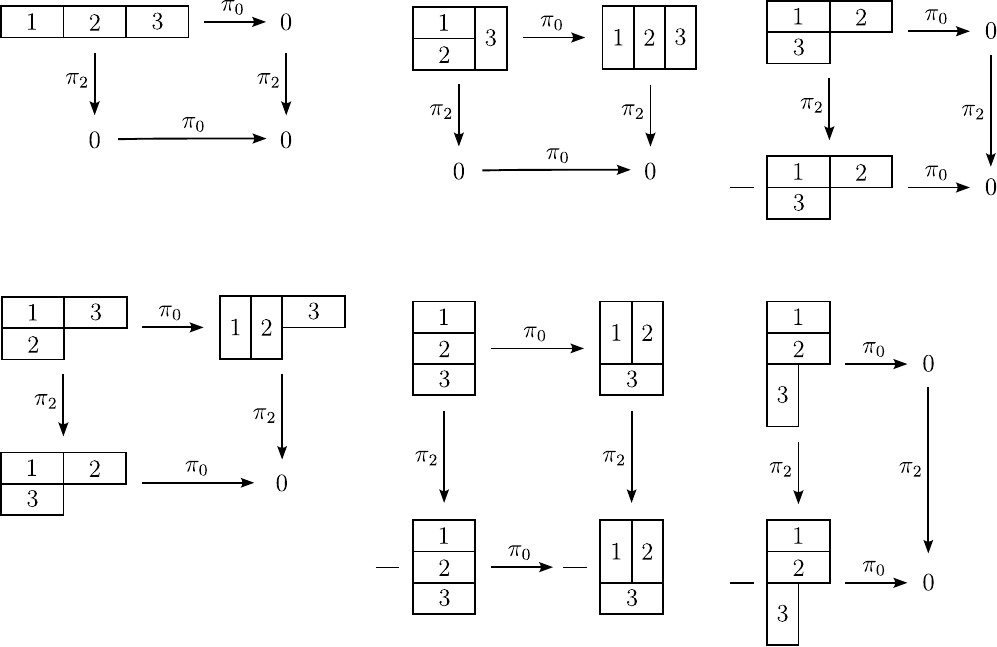}\end{array}\]

We next prove that $\pi_0\pi_1\pi_0\pi_1(T) = \pi_1\pi_0\pi_1\pi_0(T)$. The actions of $\pi_0$ and $\pi_1$ on $T$ depend only on the positions of $\dom_1(T)$ and $\dom_2(T)$. By the definition of domino tableaux, we cannot have $\pi_1(T)=s_1(T)$, so $\pi_0\pi_1\pi_0\pi_1(T)$ is either $-\pi_0\pi_1\pi_0(T)$ or $0$. If $\pi_0(T)=0$, then both $\pi_1\pi_0\pi_1\pi_0(T)$ and $\pi_0\pi_1\pi_0\pi_1(T)$ are $0$. If $\pi_0(T)=s_0(T)$, then we have $\pi_1(T)=-T$ and $\pi_1\pi_0(T)=0$, so again both $\pi_1\pi_0\pi_1\pi_0(T)$ and $\pi_0\pi_1\pi_0\pi_1(T)$ are $0$. If $\pi_0(T)=-T$, then $\pi_1\pi_0\pi_1\pi_0(T)$ and $\pi_0\pi_1\pi_0\pi_1(T)$ are both $0$ if $\pi_1(T)=0$ and are both equal to $T$ if $\pi_1(T)=-T$. 

Finally, we show that for $i\geq 1$, we have $\pi_i\pi_{i+1}\pi_i(T) = \pi_{i+1}\pi_i\pi_{i+1}(T)$. We consider three cases. 

\medskip
\noindent
{\bf Case 1.} Suppose $\pi_i(T)=0$. In this case, $\pi_i\pi_{i+1}\pi_i(T)=0$, and it follows immediately that ${\pi_{i+1}\pi_i\pi_{i+1}(T) = 0}$ if either $\pi_{i+1}(T)=0$ or $\pi_{i+1}(T)=-T$. Suppose $\pi_{i+1}(T) = s_{i+1}(T)$. Then $\dom_{i+1}(T)$ shares a row with $\dom_i(T)$, and $\dom_{i+2}(T)$ is strictly above and strictly right of $\dom_{i+1}(T)$. It follows that $\pi_i\pi_{i+1}(T) = s_is_{i+1}(T)$. The dominoes occupied by $i$ and $i+1$ in $T$ are occupied by $i+1$ and $i+2$, respectively, in $s_is_{i+1}(T)$, so $\pi_{i+1}\pi_i\pi_{i+1}(T) = \pi_{i+1}s_is_{i+1}(T) = 0$. 

\medskip
\noindent
{\bf Case 2.} Suppose $\pi_i(T)=-T$. If $\pi_{i+1}(T)=0$, then $\pi_{i}\pi_{i+1}\pi_{i}(T)$ and $\pi_{i+1}\pi_{i}\pi_{i+1}(T)$ are both $0$. If $\pi_{i+1}(T)=-T$, then $\pi_{i}\pi_{i+1}\pi_{i}(T)$ and $\pi_{i+1}\pi_{i}\pi_{i+1}(T)$ are both equal to $-T$. Now assume that $\pi_{i+1}(T)=s_{i+1}(T)$. Let $X=\pi_i(s_{i+1}(T))$. Since $\pi_{i}\pi_{i+1}\pi_{i}(T)=-X$ and ${\pi_{i+1}\pi_{i}\pi_{i+1}(T)=\pi_{i+1}(X)}$, we need to show that $\pi_{i+1}(X)=-X$. This is obvious if $X=0$. If $X=s_is_{i+1}(T)$, then we have $\dom_i(T)=\dom_{i+1}(X)$ and $\dom_{i+1}(T)=\dom_{i+2}(X)$, so $i+1\in\Des(X)$ (since $i\in\Des(T)$). In this case, $\pi_{i+1}(X)=-X$, as desired. Finally, suppose $X=-s_{i+1}(T)$. Because $\pi_{i+1}(T)=s_{i+1}(T)$, it follows from the definition of the action of $\pi_{i+1}$ that $i+1\in\Des(s_{i+1}(T))$. Consequently, ${\pi_{i+1}(X)=-\pi_{i+1}(s_{i+1}(T))=s_{i+1}(T)=-X}$. 

\medskip
\noindent
{\bf Case 3.} Suppose $\pi_i(T)=s_i(T)$. For this case, let us first assume $\pi_{i+1}(T)=0$. Then $\dom_{i+1}(T)$ is strictly above $\dom_i(T)$ but not strictly above $\dom_{i+2}(T)$. It follows that $\dom_{i+1}(s_i(T))$ is not strictly above $\dom_{i+2}(s_i(T))$, so $\pi_{i+1}(s_i(T))$ is either $0$ or $s_{i+1}s_i(T)$. If $\pi_{i+1}(s_i(T))=0$, then $\pi_{i}\pi_{i+1}\pi_{i}(T)$ and $\pi_{i+1}\pi_{i}\pi_{i+1}(T)$ are both $0$. On the other hand, if $\pi_{i+1}(s_i(T))=s_{i+1}s_i(T)$, then we have ${\dom_i(\pi_{i+1}\pi_i(T))=\dom_{i+1}(T)}$ and $\dom_{i+1}(\pi_{i+1}\pi_i(T))=\dom_{i+2}(T)$, so the assumption that $\pi_{i+1}(T)=0$ implies that $\pi_{i}\pi_{i+1}\pi_i(T)=0=\pi_{i+1}\pi_i\pi_{i+1}(T)$. 

Now assume $\pi_{i+1}(T)\neq 0$. Then $\pi_{i+1}(T)$ is either $-T$ or $s_{i+1}(T)$. These two subcases are similar, so we will only consider the subcase in which $\pi_{i+1}(T)=-T$. This means that $\dom_{i+1}(T)$ is strictly above both $\dom_{i}(T)$ and $\dom_{i+2}(T)$. Let $X=\pi_{i+1}\pi_i(T)$. If $X=0$, then $\pi_{i}\pi_{i+1}\pi_{i}(T)$ and $\pi_{i+1}\pi_{i}\pi_{i+1}(T)$ are both $0$, so we may assume $X$ is either $s_{i+1}s_i(T)$ or $-s_i(T)$. Let $T'$ be the unique tableau in the support of $X$. We have $\pi_i\pi_{i+1}\pi_i(T)=\pi_i(X)$ and $\pi_{i+1}\pi_i\pi_{i+1}(T)=-X$, so we need to show that $\pi_i(X)=-X$. Since the action of $\pi_{i+1}$ does not affect the domino with entry $i$, we have $\dom_i(T')=\dom_i(s_i(T))=\dom_{i+1}(T)$. Because $\dom_{i+1}(T)$ is strictly above $\dom_{i}(T)$ and $\dom_{i+2}(T)$, we see that $i\in\Des(T')$. It follows that $\pi_i(T')=-T'$, so $\pi_i(X)=-X$. 
\end{proof}

In what follows, we write $\mathbb C\SDT(\lambda)$ for the $H_n^B(0)$-module from \Cref{thm:dominoHecke}. 

\begin{theorem}\label{thm:dominocharacteristic}
For $\lambda\vdash 2n$, we have $\ch^B(\mathbb C\SDT(\lambda))=G_\lambda$. 
\end{theorem}

\begin{proof}
For $T,T'\in\SDT(\lambda)$, let us write $T\preceq T'$ if there is a (possibly empty) sequence $i_1,\ldots,i_r$ of indices in $[0,n-1]$ such that $T'=\pi_{i_r}\cdots\pi_{i_1}(T)$. We claim that $\preceq$ is a partial order. It is clear that $\preceq$ is reflective and transitive. It follows from the definition of the action of $\pi_0,\ldots,\pi_{n-1}$ that if $T\preceq T'$, then $\Des(T)\supseteq\Des(T')$, where equality holds if and only if $T=T'$. This implies that $\preceq$ is antisymmetric. 

Let $m=|\SDT(\lambda)|$, and let $T_1,\ldots,T_m$ be a linear extension of the poset $(\SDT(\lambda),\preceq)$. For $0\leq k\leq m$, let ${\bf N}_k=\text{span}\{T_1,\ldots,T_k\}$. Using an argument very similar to the one used in the proof of \cref{thm:ascent_compatible_B}, we find that ${\bf N}_k/{\bf N}_{k-1}$ is isomorphic to the simple $H_n^B(0)$-module $\mathcal S_{\Des(T_k)}^B$. We have the filtration $0={\bf N}_0\subseteq{\bf N}_1\subseteq\cdots\subseteq{\bf N}_m=\mathbb C\SDT(\lambda)$, so 
\[\ch^B(\mathbb C\SDT(\lambda))=\sum_{k=1}^m\ch^B({\bf N}/{\bf N}_{k-1})=\sum_{k=1}^m F_{\Des(T_k)}^B=G_\lambda. \qedhere\] 
\end{proof}

\subsection{Shifted domino tableaux}\label{subsec:shifted}

The \dfn{$2$-quotient} of a partition $\lambda=(\lambda_1, \ldots , \lambda_k)$ is the pair $(\mu,\nu)$ of partitions obtained as follows. Let $\lambda^\star$ be the partition obtained by adding $k-i$ to the part $\lambda_i$ for each $1\le i \le k$. Now let $w$ be obtained from $\lambda^\star$ by replacing the odd parts of $\lambda^\star$ with $1,3,5,\ldots$ and the even parts of $\lambda^\star$ with $0,2,4,\ldots$, from right to left. Then $\mu$ (respectively, $\nu$) is the partition found by subtracting the even (respectively, odd) entries of $w$ from the even (respectively, odd) entries of $\lambda^\star$, dividing the result by $2$, and removing any $0$ entries. 

\begin{example}
If $\lambda=(7,7,6,5,1)$, then $\lambda^\star=(11,10,8,6,1)$, $w=(3,4,2,0,1)$, $\mu=(3,3,3)$, and $\nu=(4)$. 
\end{example}

Let $\lambda$ be a partition whose $2$-quotient $(\mu,\nu)=((\mu_1, \ldots , \mu_p),(\nu_1, \ldots, \nu_q))$ satisfies $\mu_p\ge p$ and $\nu_q\ge q$. A domino tiling of $\lambda$ is a \dfn{shifted tiling} of $\lambda$ if there is no vertical domino $d$ lying on the main diagonal such that all dominoes left of $d$ and adjacent to $d$ are strictly below the main diagonal. Given a shifted tiling of $\lambda$, we define a \dfn{semistandard shifted domino tableau} of shape $\lambda$ to be a filling of the dominoes in the tiling lying weakly above the main diagonal with entries from the totally ordered set $\{0<1'<1<2'<2<\cdots\}$ such that
\begin{enumerate}
    \item entries weakly increase from left to right along rows and from top to bottom in columns;
    \item each row contains at most one $i'$, and each column contains at most one $i$;
    \item the northwest-most domino may be filled with $0$ only if it is horizontal.
\end{enumerate}
This definition is a modification of the one given by Chemli \cite{Chemli}. Denote the set of semistandard shifted domino tableaux of shape $\lambda$ by $\SSShDT(\lambda)$. For $T\in \SSShDT(\lambda)$, let $\wt(T)$ denote the \dfn{weight} of $T$, which we define to be the vector $(\wt_0(T),\wt_1(T),\ldots)$, where $\wt_i(T)$ is the number of entries in $T$ equal to $i$ or $i'$ (in particular, $\wt_0(T)$ is the number of $0$ entries in $T$). Let us write $x^{\wt(T)}=x_0^{\wt_0(T)}x_1^{\wt_1(T)}\cdots$. We define the \dfn{shifted domino function} indexed by $\lambda$ to be
\[H_\lambda = \sum_{T\in \SSShDT(\lambda)}x^{\wt(T)}.\]

\begin{figure}[ht]
  \begin{center}{\includegraphics[height=2.493cm]{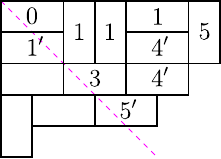}\qquad\qquad\qquad\qquad\includegraphics[height=2.493cm]{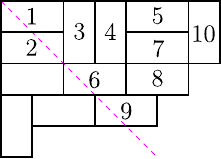}}
  \end{center}
  \caption{A semistandard shifted domino tableau $T\in\SSShDT(\lambda)$ (left) and a standard shifted domino tableau $U\in\SShDT(\lambda)$ (right), where $\lambda=(7,7,6,5,1)$. We have $\wt(T)=(1,4,0,1,2,2)$ and $\Des(U)=\{1,5,7,8\}$.}\label{Fig2}
\end{figure}

Define a \dfn{standard shifted domino tableau} of shape $\lambda$ to be a bijective filling of the dominoes weakly above the main diagonal in a shifted tiling of $\lambda$ with entries in $[m]$ (where $m$ is the number of dominoes weakly above the main diagonal) such that entries strictly increase from left to right along rows and from top to bottom in columns. See the right of \Cref{Fig2}. Let $\SShDT(\lambda)$ denote the set of standard shifted domino tableaux of shape $\lambda$. For $Q\in\SShDT(\lambda)$, let $\dom_i(Q)$ be the domino filled with the entry $i$ in $Q$. A number $i\in [0,m-1]$ is a \dfn{descent} of $Q$ if $i=0$ and $\dom_1(Q)$ is vertical or if $i>0$ and $\dom_{i+1}(Q)$ is strictly lower than $\dom_i(Q)$. Let $\Des(Q)$ be the set of descents of $Q$. In \Cref{cor:H_lambda}, we will prove that
\[H_\lambda = \sum_{Q\in \SShDT(\lambda)}\Delta^B(\Des(Q)),\] 
where $\Delta^B$ is as defined in \eqref{eq:Delta}.

Define a \dfn{marked standard shifted domino tableau} to be a standard shifted domino tableau in which each entry is either primed or unprimed. (So a standard shifted domino tableau with $n$ dominoes yields $2^n$ marked standard shifted domino tableaux.) Let $\MSShDT(\lambda)$ denote the set of marked standard shifted domino tableaux of shape $\lambda$. Given $S\in \MSShDT(\lambda)$, let $\demark(S)\in\SShDT$ be the standard shifted domino tableau obtained by removing all primes from $S$. We write $\dom_i(S)$ for the domino in $S$ that contains $i$ or $i'$ (so $\dom_i(S)=\dom_i(\demark(S))$). We say an index $i\in[0,n-1]$ is a \dfn{descent} of $S\in \MSShDT(\lambda)$ if one of the following holds:
\begin{itemize}
\item $i=0$, and either $1'$ is in $S$ or $\dom_1(S)$ is vertical (or both);
\item $i>0$, $i$ is in $S$, and $i\in\Des(\demark(S))$; 
\item $i>0$, $(i+1)'$ is in $S$, and $i\not\in\Des(\demark(S))$.
\end{itemize}
Let $\Des(S)$ denote the set of descents of $S$. 

\begin{lemma}\label{lem:descentininterval}
Let $S\in \MSShDT(\lambda)$. Suppose there exist $i,j\in[n]$ with $i<j$ such that one of the following holds: 
\begin{enumerate}[\normalfont(I)]
\item\label{item:descentinterval1} $i$ and $j'$ are in $S$; 
\item\label{item:descentinterval2} $i$ and $j$ are in $S$, and $\dom_j(S)$ is strictly lower than $\dom_i(S)$; 
\item\label{item:descentinterval3} $i'$ and $j'$ are in $S$, and $\dom_i(S)$ is weakly below $\dom_j(S)$.
\end{enumerate}
Then there is a descent of $S$ in $[i,j-1]$.
\end{lemma}
\begin{proof}
Suppose by way of contradiction that there exist $i<j$ satisfying one of the conditions \eqref{item:descentinterval1}, \eqref{item:descentinterval2}, \eqref{item:descentinterval3} such that there are no descents of $S$ in $[i,j-1]$. We may choose such $i$ and $j$ so as to minimize the quantity $j-i$. 

Suppose $i$ and $j$ satisfy condition \eqref{item:descentinterval1}. Then there must exist $k\in[i,j-1]$ such that $k$ and $(k+1)'$ are in $S$. But then $k$ must be in $\Des(S)$, which is a contradiction. 

Now suppose $i$ and $j$ satisfy condition \eqref{item:descentinterval2}. For every $k\in[i+1,j-1]$, we know that $k$ is in $S$ since, otherwise, $i$ and $k$ would satisfy condition \eqref{item:descentinterval1}, contradicting the minimality of $j-i$. Since $k$ is not a descent of $S$, it is not a descent of $\demark(S)$. Thus, $\dom_{k+1}(\demark(S))$ is not strictly lower than $\dom_k(\demark(S))$. From the definition of a standard shifted domino tableau, the only way that part of $\dom_{k+1}(\demark(S))$ can be strictly lower than part of $\dom_k(\demark(S))$ is if $\dom_k(\demark(S))$ is horizontal, $\dom_{k+1}(\demark(S))$ is vertical and immediately right of $\dom_{k}(\demark(S))$, and the main diagonal intersects both of these dominoes. To satisfy condition \eqref{item:descentinterval2}, both parts of $\dom_j(\demark(S))$ must be strictly below both parts of $\dom_i(\demark(S))$. This means the configuration described has to occur for at least two different values of $k$ between $i$ and $j$. But in any two occurrences of this configuration, both parts of each domino in the later occurrence are strictly below both parts of each domino in the earlier occurrence. Therefore, between two occurrences of this configuration, there must be some $\ell$ for which both parts of $\dom_{\ell+1}(\demark(S))$ are strictly below both parts of $\dom_{\ell}(\demark(S))$, which contradicts the assumption that there are no descents of $S$ in $[i, j-1]$.

Finally, suppose $i$ and $j$ satisfy condition \eqref{item:descentinterval3}. For every $k\in[i,j-1]$, we know that $(k+1)'$ is in $S$ since, otherwise, $k+1$ and $j$ would satisfy condition \eqref{item:descentinterval1}, contradicting the minimality of $j-i$. Since $k$ is not a descent of $S$ and $(k+1)'$ is in $S$, we must have $k\in\Des(\demark(S))$. Thus, $\dom_{k+1}(\demark(S))$ is strictly lower than $\dom_k(\demark(S))$. As this is true for all $k\in[i,j-1]$, it follows that $\dom_j(\demark(S))$ is strictly lower $\dom_i(\demark(S))$, which contradicts our assumption that $i$ and $j$ satisfy condition \eqref{item:descentinterval3}. 
\end{proof}

Given $T\in \SSShDT(\lambda)$, let $\stand(T)\in \MSShDT(\lambda)$ be the \dfn{standardization} of $T$ obtained as follows. Let $n_1 < n_2 < \cdots<n_r$ be the positive integers $i$ such that $\wt_i(T)\neq 0$. Replace the $0$ entries in $T$ from left to right with integers $1, 2, \ldots , \wt_0(T)$. Then replace the entries $n_1'$ and $n_1$ in $T$ with $\wt_0(T)+1, \ldots , \wt_0(T)+\wt_{n_1}(T)$ or their primed versions (replacing primed entries with primed integers and unprimed entries with unprimed integers) in the following order: first replace entries $n_1'$ from top to bottom; then replace entries $n_1$ from left to right. This is well defined since all $n_1'$ and $n_1$ entries appear either strictly below or strictly right of each $0$ entry in $T$, the entries $n_1'$ form a vertical strip (no two dominoes in the same row both have entry $n_1'$), the entries $n_1$ form a horizontal strip (no two dominoes in the same column both have entry $n_1$), and every $n_1$ is either strictly below or strictly right of each $n_1'$. Then continue in the same manner with $n_2'$ and $n_2$, and so on. 

We now show that the semistandard shifted domino tableaux with a prescribed stardardization $S\in \MSShDT(\lambda)$ generate the type-$B$ fundamental quasisymmetric function corresponding to the descent set of $S$.

\begin{theorem}\label{thm:stand}
Let $S\in \MSShDT(\lambda)$. Then
\[\sum_{\substack{T\in \SSShDT(\lambda) \\ \stand(T)=S}}\!\!\!\!\!x^{\wt(T)} = F^B_{\Des(S)}.\]
\end{theorem}
\begin{proof}
Suppose $\stand(T)=S$, and let $n_1 < n_2 < \cdots< n_r$ be the positive integers $i$ such that $\wt_i(T)\neq 0$. The semistandard shifted domino tableau $T$ is obtained from $S$ by replacing the entries $1,2,\ldots,\wt_0(T)$ with $0$, replacing the entries $\wt_0(T)+1, \ldots , \wt_0(T)+ \wt_{n_1}(T)$ with $n_1$, and so on, and keeping the primes in the same dominoes as in $S$. Let $T_i$ denote the entry that fills the domino $\dom_i(S)$ in the tableau $T$, and let $|T_i|$ be the unprimed version of $T_i$ (so if $T_i\in \{k',k\}$, then $|T_i|=k$). 

We first show that $\{\wt_0(T), \wt_0(T)+\wt_{n_1}(T), \ldots,\wt_0(T)+\wt_{n_1}(T)+\cdots+\wt_{n_r}(T)\}$ contains $\Des(S)$; this is equivalent to showing that the monomial $x^{\wt(T)}$ appears in $F_{\Des(S)}^B$. 
If $0\in\Des(S)$, then the northwest-most domino is vertical or primed, hence cannot have the entry $0$ in $T$, so $\wt_0(T)=0$. Now consider a descent $i>0$ of $S$. We need to show that $|T_i|<|T_{i+1}|$. Suppose for a contradiction that $|T_i|=|T_{i+1}|$. If $\dom_i(S)$ is strictly above $\dom_{i+1}(S)$, then $i$ must be in $S$, and $\dom_{i+1}(S)$ must be weakly left of $\dom_i(S)$. Then by definition of the standardization, the entry $T_{i+1}$ (regardless of whether it is primed or unprimed) is standardized before $T_i$, contradicting the fact that $\stand(T)=S$. If $(i+1)'$ is in $S$, then $\dom_{i+1}(S)$ is weakly above and strictly right of $\dom_i(S)$, and it follows that the entry $T_{i+1}$ is standardized before $T_i$ (regardless of whether $T_i$ is primed or unprimed), again contradicting the fact that $\stand(T)=S$. Hence, if $\stand(T)=S$, then $x^{\wt(T)}$ is a term in $F^B_{\Des(S)}$.

Now, fix a term $x^b = x_0^{b_0}x_{n_1}^{b_{n_1}}x_{n_2}^{b_{n_2}}\cdots$ in $F^B_{\Des(S)}$. Let $T$ be obtained by replacing entries $1,2,\ldots , b_0$ in $S$ with $0$, then replacing entries $b_0+1, \ldots , b_0+ b_{n_1}$ with $n_1$, and so on, keeping the primes in the same dominoes as in $S$. By definition, $\wt(T)=x^b$. We need to show that $T$ is the unique element of $\SSShDT(\lambda)$ whose standardization is $S$. Note that since $x^b$ is a term in $F^B_{\Des(S)}$, the descent set of the type-$B$ composition $(b_0,b_{n_1},b_{n_2}, \ldots)$ contains $\Des(S)$. In particular, if $i>0$ is a descent of $S$, then $|T_i|<|T_{i+1}|$, and if $\Des(S)$ contains $0$, then $b_0=0$. 

We first show that $T\in \SSShDT(\lambda)$. By definition, for each $i\ge 0$, every domino containing $n_i$ or $n_i'$ in $T$ appears either strictly right of or strictly below every domino containing $n_{i+1}$ or $n_{i+1}'$, so entries weakly increase along rows and columns. We need to check the conditions on $0$ entries and on primed and unprimed entries with the same value in rows and columns. If $b_0=0$, there are no $0$ entries in $T$ to check. If $b_0>0$, then $0\notin\Des(S)$, so the northwest-most domino is horizontal. If a $0$ was placed in a primed domino in $T$, then the first instance of a $0$ being placed in a primed domino follows a $0$ being placed in an unprimed domino to the left (or above), so these dominoes correspond to a descent in $S$, contradicting the fact that they both receive the value $0$. If a $0$ was placed in a domino strictly below a domino with entry $0$, then there must have been an instance of a $0$ being placed in a domino strictly below the domino in which the previous $0$ was placed, and since both of these dominoes are unprimed, these dominoes correspond to a descent in $S$, contradicting the fact that they both received $0$ entries.  

Now suppose that the dominoes containing $T_k$ and $T_{k+\ell}$ in $T$ violate the conditions defining semistandard shifted domino tableaux, where $|T_k|=|T_{k+\ell}| = n_i$. (Note that the domino containing $T_{k+\ell}$ is strictly right of or strictly below the domino containing $T_k$.) Such a violation means that one of the following must hold: 
\begin{itemize}
\item $T_k=n_i$ and $T_{k+\ell}=n_i'$, and these dominoes are in the same row or column; 
\item $T_k=n_i$ and $T_{k+\ell}=n_i$, and these dominoes are in the same column; 
\item $T_k=n_i'$ and $T_{k+\ell}=n_i'$, and these dominoes are in the same row. 
\end{itemize}   
By \Cref{lem:descentininterval}, there is a descent of $S$ in the interval $[k, k+\ell-1]$, which contradicts the fact that $|T_j|=n_i$ for all $k<j<k+\ell$. 

Next, observe that only $T$ could possibly standardize to $S$. This is because the standardization procedure replaces all entries $i$ in $K\in \SSShDT(\lambda)$ with integers that are smaller than those with which it replaces entries $i+1$. So if $K$ standardizes to $S$, all $0$ entries in $K$ must be in the same dominoes as $1, \ldots , b_0$ in $S$, all $n_1$ entries in $K$ must be in the same dominoes as $b_0+1, \ldots , b_0+b_{n_1}$ in $S$, etc, and the position of the primes is determined by $S$. So $K=T$. 

Finally we must show that indeed $\stand(T)=S$. We must ensure that if $|T_k|=|T_{k+1}|=j$, then the entry $T_k$ is standardized before the entry $T_{k+1}$. Note that $|T_k|=|T_{k+1}|$ implies $k\notin \Des(S)$. Note that $k$ and $(k+1)'$ cannot both be in $S$ since this would force $k$ to be a descent of $S$. If $k'$ and $k+1$ are in $S$, then by definition the standardization procedure encounters $T_k=j'$ before $T_{k+1}=j$. If $k$ and $k+1$ are both in $S$, then since $k\notin \Des(S)$, $\dom_{k+1}(S)$ is strictly right of $\dom_k(S)$. Hence, $T_k=j$ is standardized before $T_{k+1}=j$.  If $k'$ and $(k+1)'$ are both in $S$, then since $k\notin \Des(S)$, $\dom_{k+1}(S)$ is strictly below $\dom_k(S)$. Hence, $T_k=j'$ is standardized before $T_{k+1}=j'$.
\end{proof}

\begin{theorem}\label{thm:peak}
Let $\lambda$ be a partition of $2n$, and fix $Q\in \SShDT(\lambda)$. Then
\[\sum_{\substack{S\in \MSShDT(\lambda) \\ \demark(S)=Q} } F^B_{\Des(S)} = \Delta^B(\Des(Q)).\]
\end{theorem}
\begin{proof}
Recall the notation from \cref{subsec:quasisymmetric}. Let us write $\Peak^B(Q)$ and $\Val^B(Q)$ for the type-$B$ peak set $\Peak^B(\Des(Q))$ and the type-$B$ valley set $\Val^B(\Des(Q))$, respectively. Our goal is to show that
\[\sum_{\substack{S\in \MSShDT(\lambda) \\ \demark(S)=Q}} F^B_{\Des(S)} = \sum_{\substack{D\subseteq [0,n-1] \\ \Peak^B(Q)\subseteq D\triangle (D+1)}}2^{|\Val^B(Q)|}F^B_D.\] 
Note that both sides of this equation have $2^n$ terms. 
We will first show that any $S\in \MSShDT(\lambda)$ such that $\demark(S)=Q$ satisfies $\Peak^B(Q)\subseteq \Des(S)\triangle (\Des(S)+1)$. We will then show that for every $D\subseteq [0, n-1]$, the set $\Gamma_D=\{S\in\MSShDT(\lambda):\demark(S)=Q, \Des(S)=D\}$ has size $2^{|\Val^B(Q)|}$. 

Suppose $S\in \MSShDT(\lambda)$ and $\demark(S)=Q$. Consider $i\in[n-1]$. Recall that $i\in \Peak^B(Q)$ if and only if $i\in \Des(Q)$ and $i-1\notin \Des(Q)$. If $i-1\notin \Des(Q)$, then $i-1\in \Des(S)$ if and only if $i'$ is in $S$. Also, if $i\in \Des(Q)$, then $i\in \Des(S)$ if and only if $i$ is in $S$. Suppose $i\in \Peak^B(Q)$. If $i'$ is in $S$, then $i-1\in \Des(S)$ and $i\notin \Des(S)$, so $i\in \Des(S)\triangle (\Des(S)+1)$. If $i$ is in $S$, then $i-1\notin \Des(S)$ and $i\in \Des(S)$, so again $i\in \Des(S)\triangle (\Des(S)+1)$. This shows that $\Peak^B(Q)\subseteq \Des(S)\triangle (\Des(S)+1)$.

Now suppose $D\subseteq [0,n-1]$ is such that $\Peak^B(Q)\subseteq D\triangle(D+1)$, and define $\Gamma_D$ as above. First, we show $\Gamma_D$ is nonempty. Suppose $\Peak^B(Q) = \{i_1, \ldots , i_r\}$. The condition $\Peak^B(Q)\subseteq D\triangle (D+1)$ implies that for each $1\le j \le r$, precisely one of $i_j -1$ or $i_j$ is in $D$. We construct $S\in \MSShDT(\lambda)$ such that $\demark(S)=Q$ and $\Des(S)=D$ by considering the entries in $Q$ from smallest to largest and choosing either to prime them or leave them unprimed. First, we consider the entries of $Q$ in $[1, i_1]$. We have two cases depending on whether $i_1=1$ or $i_1>1$.

If $i_1=1$, then we have $1\in \Des(Q)$ and $0\notin \Des(Q)$. If $0\in D$, then we prime $1$, so $0\in\Des(S)$. Note that since $i_1=1$ and $0\in D$, we must have $1\notin D$; moreover, $1\notin \Des(S)$ since $1\in \Des(Q)$ and $1$ is primed. If $0\notin D$, then we do not prime $1$. Since $0\notin D$ and $i_1=1$, we must have $1\in D$, and we will also have $1\in \Des(S)$ since $1\in \Des(Q)$ and $1$ is unprimed.

Now assume $i_1>1$. By definition of the peak set of $Q$, any descents of $Q$ that are smaller than $i_1$ form an initial segment in $[0, i_1-2]$. Furthermore, since $0\in \Des(Q)$ implies $\dom_1(Q)$ is vertical, we have $1\notin \Des(Q)$ whenever $0\in \Des(Q)$. Hence, no number smaller than $i_1$ can be a descent of $Q$, except possibly $0$. Now, prime $1$ if both $0\in D$ and $\dom_1(Q)$ is horizontal, and leave $1$ unprimed otherwise. This ensures that $0\in \Des(S)$ if and only if $0\in D$. Next, for all $k\in D\cap [1, i_1-1]$, prime the entry $k+1$ in $Q$, and leave all other entries unprimed. Since there are no descents of $Q$ in the interval $[1, i_1-1]$, this ensures that, for $k\in [1, i_1-1]$, we will have $k\in \Des(S)$ if and only if $k\in D$. Finally, consider $i_1$, which by definition is in $\Des(Q)$. If $i_1-1\in D$, then we have already primed $i_1$, and therefore $i_1$ will not be a descent of $S$, but this is precisely what is required since $i_1-1\in D$ implies $i_1\notin D$. On the other hand, if $i_1-1\notin D$, then we have already chosen $i_1$ to be unprimed, and then since $i_1\in \Des(Q)$, we have $i_1\in \Des(S)$. But again this is precisely what is required since $i_1-1\notin D$ implies $i_1\in D$.

Now consider the entries of $Q$ in the interval $[i_1+1, i_2]$. The descents of $Q$ in this interval form a (possibly empty) initial segment in this interval, say $[i_1+1, s_1]$, where $s_1 \le i_2-2$. Leave unprimed the elements of $[i_1+1, \ldots s_1]$ that are in $D$, and prime the elements of $[i_1+1, \ldots s_1]$ that are not in $D$. The descents of $S$ in this interval will then be precisely the elements of $D$ in this interval. The entries in $[s_1+1 \ldots i_2-1]$ are not descents of $Q$. For each $k$ in this interval, prime $k+1$ if $k\in D$, and leave $k+1$ unprimed if $k\not\in D$. Then the descents of $S$ in this interval are precisely the elements of $D$ in this interval. Continue in this manner with the elements in the interval $[i_2+1, i_3]$ and all subsequent intervals: the resulting $\MSShDT$ $S$ satisfies $\demark(S)=Q$ and $\Des(S)=D$. 

Let $S^*\in\Gamma_D$. We claim that $\Gamma_D$ is equal to the set of tableau that can be obtained from $S^*$ by priming or unpriming some of the entries in $\Val^B(Q)$. Suppose $m\in\Val^B(Q)$. Then $m\notin \Des(Q)$, so whether $m$ is a descent in $S^*$ is entirely controlled by whether $m+1$ is primed or not, regardless of whether $m$ is primed. Similarly, since $m-1\in \Des(Q)$, whether $m-1$ is a descent in $S^*$ is entirely controlled by whether $m-1$ is primed or not, regardless of whether $m$ is primed. Thus, priming or unpriming $m$ does not change the descent set. On the other hand, every integer $p\in[n]\setminus\Val^B(Q)$ is such that either $p\in \Des(Q)$ (in which case priming or unpriming $p$ changes whether $p+1$ is in the descent set of the tableau), or $p-1\notin\Des(Q)$ (in which case priming or unpriming $p$ changes whether $p-1$ is in the descent set of the tableau). 
\end{proof}

From \cref{thm:stand,thm:peak} we conclude the following.

\begin{corollary}\label{cor:H_lambda}
The shifted domino function $H_\lambda$ satisfies
\[H_\lambda = \sum_{Q\in \SShDT(\lambda)}\Delta^B(\Des(Q)).\]
\end{corollary}

\section{0-Hecke--Clifford Algebras in Type~$B$}\label{sec:0HeckeClifford}

The \dfn{Clifford algebra} $Cl_n$ is the $\mathbb C$-algebra with generators $c_1,\ldots, c_n$ subject to the relations ${c_j^2 = -1}$ and $c_ic_j=-c_jc_i$ for $i\neq j$. The \dfn{$0$-Hecke--Clifford algebra} $HCl_n(0)$ is the $\mathbb C$-algebra generated by the (type-$A$) $0$-Hecke algebra $H_n(0)$ and the Clifford algebra $Cl_n$ such that the generators $\pi_1,\ldots,\pi_{n-1},c_1,\ldots,c_n$ satisfy the additional relations
\begin{align*}
\pi_ic_j & = c_j\pi_i \,\,\,\, \mbox{ for } j\neq i, i+1;  \\
\pi_ic_i & =  c_{i+1}\pi_i; \\
(\pi_i+1)c_{i+1} & = c_i(\pi_i+1).
\end{align*}

We now introduce the \dfn{type-$B$ 0-Hecke--Clifford algebra} $HCl_n^B(0)$ as the algebra generated by the $0$-Hecke algebra $H_n^B(0)$ and the Clifford algebra $Cl_n$ such that the generators $\pi_0,\ldots,\pi_{n-1},c_1,\ldots,c_n$ satisfy the additional relations
\begin{align*}
\pi_ic_j & = c_j\pi_i \hspace{1.52cm} \mbox{ for } j\neq i, i+1;  \\
\pi_ic_{i+1} & =  c_{i}\pi_i \hspace{1.55cm} \mbox{ for } 1\le i \le n-1; \\
(\pi_i+1)c_i & = c_{i+1}(\pi_i+1) \hspace{.2cm} \mbox{ for } 1\le i \le n-1; \\
\pi_0c_1 & = \sqrt{-1}\pi_0.
\end{align*}
Note the differences between these relations and the relations defining the $0$-Hecke--Clifford algebra in type~$A$. These differences end up being necessary in order to make the algebra and combinatorics work out correctly. 

We now study the structure of the $HCl_n^B(0)$-modules induced from the simple $H_n^B(0)$-modules. We give a tableau description of the basis elements of these modules and show that the type-$B$ quasisymmetric characteristics of the restrictions of these modules to $H_n^B(0)$ are the type-$B$ peak functions, thus giving a representation-theoretic interpretation of such functions. Moreover, we show that the isomorphism classes of these induced modules are indexed by the type-$B$ peak sets. We proceed in a manner similar to \cite[Sections 5.2 and 5.3]{BHT}, where analogous results are obtained in type $A$.

For each subset $D=\{d_1<\cdots<d_r\}\subseteq[n]$, we write $c_D$ for the element $c_{d_1}\cdots c_{d_r}$ (with indices in increasing order). For each $I\subseteq[0,n-1]$, let us fix a nonzero vector 
$\varepsilon_I$ in the simple $H_n^B(0)$-module $\mathcal S_I^B$ so that $\mathcal S_I^B=\mathbb C\{\varepsilon_I\}$. Recall that the module structure is defined by \[\pi_j\cdot\varepsilon_I = \begin{cases} -\varepsilon_I & \mbox{ if } j\in I \\
                                       0 & \mbox{ if } j\notin I.
\end{cases}\]
Let \[\mathcal M_I=\text{Ind}_{H_n^B(0)}^{HCl_n^B(0)}(\mathcal S_I^B)\] be the $HCl_n^B(0)$-module induced by $\mathcal S_I^B$. The set $\{c_D\varepsilon_I:D\subseteq[n]\}$ is a linear basis for $\mathcal M_I$. 

It will be convenient to represent a subset $I\subseteq[0,n-1]$ by a ribbon shape with $n+1$ boxes filled with $0,1,\ldots,n$ such that box filled with $i+1$ is immediately below the box filled with $i$ if $i\in I$ and is immediately left of the box filled with $i$ if $i\not\in I$. We can then represent the element $c_D\varepsilon_I$ of $\mathcal M_I$ by replacing each entry $d\in D$ with $\overline d$. For example, if $n=7$, then \[\ytableausetup{centertableaux, boxsize=2em}
\begin{ytableau}
0 \\
\overline 1 & \overline 2 & 3\\
\none & \none & 4\\ 
\none & \none & \overline 5  & 6\\ 
\none & \none & \none & 7
\end{ytableau}\] represents $c_{\{1,2,5\}}\varepsilon_{\{0,3,4,6\}}$. 

To describe the module structure of $\mathcal M_I$, we need to see how each generator $\pi_i$ acts on each basis element $c_D\varepsilon_I$. Representing this basis element via a filled ribbon as above, we find that $\pi_i\cdot c_D\varepsilon_I$ only depends on the boxes filled with $i$ and $i+1$; hence, we can describe the action by drawing these two boxes and ignoring all others. For $1\leq i\leq n-1$, we have \begin{alignat*}{3}
\ytableausetup{centertableaux, boxsize=2em}
&\pi_i\cdot\begin{ytableau}
{\scriptstyle i} & {\scriptstyle i+1}
\end{ytableau}= 0; \qquad\qquad\qquad\qquad 
&&\pi_i\cdot\begin{ytableau}
{\scriptstyle \overline i} & {\scriptstyle i+1}
\end{ytableau}= -\,
\begin{ytableau}
{\scriptstyle \overline i} & {\scriptstyle i+1}
\end{ytableau}+\begin{ytableau}
{\scriptstyle i} & {\scriptstyle \overline{i+1}}
\end{ytableau}\,; \\ 
&\pi_i\cdot\begin{ytableau}
{\scriptstyle i} & {\scriptstyle \overline{i+1}}
\end{ytableau}= 0; \qquad\qquad\qquad\qquad 
&&\pi_i\cdot\begin{ytableau}
{\scriptstyle \overline i} & {\scriptstyle \overline{i+1}}
\end{ytableau}= -\,
\begin{ytableau}
{\scriptstyle \overline i} & {\scriptstyle \overline{i+1}}
\end{ytableau}-\begin{ytableau}
{\scriptstyle i} & {\scriptstyle i+1}
\end{ytableau}\,; \\ 
&\pi_i\cdot\begin{ytableau}
{\scriptstyle i} \\ {\scriptstyle \overline{i+1}}
\end{ytableau}= -\,\begin{ytableau}
{\scriptstyle \overline i} \\ {\scriptstyle i+1}
\end{ytableau}\,; \qquad\qquad\qquad\qquad 
&&\pi_i\cdot\begin{ytableau}
{\scriptstyle i} \\ {\scriptstyle i+1}
\end{ytableau}= -\,\begin{ytableau}
{\scriptstyle i} \\ {\scriptstyle i+1}
\end{ytableau}\,; \\ 
&\pi_i\cdot\begin{ytableau}
{\scriptstyle \overline i} \\ {\scriptstyle \overline{i+1}}
\end{ytableau}= -\,\begin{ytableau}
{\scriptstyle i} \\ {\scriptstyle i+1}
\end{ytableau}\,; \qquad\qquad\qquad\qquad 
&&\pi_i\cdot\begin{ytableau}
{\scriptstyle \overline i} \\ {\scriptstyle i+1}
\end{ytableau}= -\,\begin{ytableau}
{\scriptstyle \overline i} \\ {\scriptstyle i+1}
\end{ytableau}\,.
\end{alignat*}
Moreover, we have 
\begin{align*}
\ytableausetup{centertableaux, boxsize=2em}
&\pi_0\cdot\begin{ytableau}
0 & 1
\end{ytableau}= 0;  \hspace{10.2cm}\\ 
&\pi_0\cdot\begin{ytableau}
0 & \overline 1
\end{ytableau}= 0; \\ 
&\pi_0\cdot\begin{ytableau}
0 \\ \overline{1}
\end{ytableau}= -\sqrt{-1}\,\begin{ytableau}
0 \\ 1
\end{ytableau}\,; \hspace{2cm}
\pi_0\cdot\begin{ytableau}
0 \\ 1
\end{ytableau}= -\,\begin{ytableau}
0 \\ 1
\end{ytableau}\,.
\end{align*}

For a fixed $I\subseteq[0,n-1]$, we define a partial order $\leq_I$ on the basis $\{c_D\varepsilon_I:D\subseteq[n]\}$ via the covering relations of the form \[\,\begin{ytableau}
\scriptstyle i & \scriptstyle \overline{i+1}
\end{ytableau}\,\lessdot_I\begin{ytableau}
\scriptstyle \overline i & \scriptstyle i+1
\end{ytableau}\,,\qquad \,\begin{ytableau}
\scriptstyle i & \scriptstyle i+1
\end{ytableau}\,\lessdot_I\begin{ytableau}
\scriptstyle \overline i & \scriptstyle \overline{i+1}
\end{ytableau}\,,\] \[ \,\begin{ytableau}
\scriptstyle \overline i \\ \scriptstyle i+1
\end{ytableau}\,\lessdot_I\begin{ytableau}
\scriptstyle i \\ \scriptstyle \overline{i+1}
\end{ytableau}\,, \qquad  \,\begin{ytableau}
\scriptstyle i \\ \scriptstyle i+1
\end{ytableau}\,\lessdot_I\begin{ytableau}
\scriptstyle \overline i \\ \scriptstyle \overline{i+1}
\end{ytableau}\,\]
for $1\leq i\leq n-1$ and the additional cover relation \[\,\begin{ytableau}
0 \\ 1
\end{ytableau}\,\lessdot_I\begin{ytableau}
0 \\ \overline 1
\end{ytableau}\,.\] In this notation, each cover relation occurs when the ribbon diagrams agree except in the displayed boxes. For example, the first of these cover relations should be interpreted as saying $c_D\varepsilon_I\lessdot_I c_{D'}\varepsilon_I$ whenever $i\notin I$, $i\in D$, $i+1\not\in D$, and $D'=(D\setminus\{i\})\cup\{i+1\}$. The final additional cover relation says that $c_D\varepsilon_I\lessdot_I c_{D'}\varepsilon_I$ whenever $0\in I$, $1\in D$, and $D'=D\setminus\{1\}$. 

The following lemma is immediate from the above discussion. 

\begin{lemma}\label{lem:triangular}
Fix $I\subseteq[0,n-1]$. For each $0\leq i\leq n-1$, the action of $\pi_i$ on the basis $\{c_D\varepsilon_I:D\subseteq[n]\}$ of $\mathcal M_I$ is triangular with respect to the partial order $\leq_I$. More precisely, for each $D\subseteq[n]$, we have \[\pi_i\cdot c_D\varepsilon_I=\alpha(i,I,D)c_D\varepsilon_I+\sum_{c_{D'}\varepsilon_I<_I c_D\varepsilon_I}\beta(i,I,D')c_{D'}\varepsilon_I\] for some $\alpha(i,I,D)\in\{0,-1\}$ and some $\beta(i,I,D')\in\mathbb C$. 
\end{lemma}

Upon inspection, we find that the eigenvalue $\alpha(i,I,D)$ is equal to $-1$ if and only if one of the following conditions holds: 
\begin{itemize}
\item $i\notin I$ and $i\in D$; 
\item $i\in I$ and $i+1\notin D$.
\end{itemize}
This implies that if $v$ is in the type-$B$ valley set $\Val^B([0,n-1]\setminus I)=([n]\cap I)\setminus(I+1)$, then none of the eigenvalues $\alpha(i,I,D)$ depend on whether or not $v$ belongs to $D$. In other words, we have the following lemma. 
\begin{lemma}\label{lem:valleys}
If $v\in\Val^B([0,n-1]\setminus I)$, then \[\alpha(i,I,D)=\alpha(i,I,D\cup\{v\})\] for every $i\in[0,n-1]$ and every $D\subseteq [n]$. 
\end{lemma}

Fix a linear extension $c_{D_1}\varepsilon_I,\ldots,c_{D_{2^n}}\varepsilon_I$ of the partial order $\leq_I$. For each $0\leq k\leq 2^n$, it follows from \Cref{lem:triangular} that the vector space $\mathcal M_I^{(k)}=\mathbb C\{c_{D_1}\varepsilon_I,\ldots,c_{D_k}\varepsilon_I\}$ is an $H_n^B(0)$-module. Thus, \[\{0\}=\mathcal M_I^{(0)}\subseteq\mathcal M_I^{(1)}\subseteq\cdots\subseteq\mathcal M_I^{(2^n)}=\Res_{H_n^B(0)}(\mathcal M_I)\] is a composition series for the $H_n^B(0)$-module obtained by restricting $\mathcal M_I$. 

For each $1\leq k\leq 2^n$, the quotient $\mathcal M_I^{(k)}/\mathcal M_I^{(k-1)}$ is a simple $H_n^B(0)$-module, so there is a unique subset $K(k,I)\subseteq [0,n-1]$ such that $\mathcal M_I^{(k)}/\mathcal M_I^{(k-1)}$ is isomorphic to $\mathcal S_{K(k,I)}^B=\mathbb C\{\varepsilon_{K(k,I)}\}$. According to \Cref{lem:triangular}, each generator $\pi_i$ acts on $\mathcal M_I^{(k)}/\mathcal M_I^{(k-1)}$ via multiplication by $\alpha(i,I,D_k)$. Therefore, it follows from the explicit description of the module structure of $\mathcal S_{K(k,I)}^B$ that 
\begin{equation}\label{eq:K(k,I)}
K(k,I)=\{i\in[0,n-1]:\alpha(i,I,D_k)=-1\}=(D_k\setminus I)\cup(I\setminus (D_k-1)).
\end{equation}

The next theorem allows us to interpret Petersen's type-$B$ peak functions as the type-$B$ quasisymmetric characteristics of the restrictions of the modules $\mathcal M_I$, thereby providing a type-$B$ analogue of \cite[Theorem~5.3]{BHT}. We remark that while the paper \cite{BHT} works directly with a set $I\in[n-1]$, we must take the complement of our set $I$ in $[0,n-1]$ in order to deal with the slight differences between the relations defining $HCl_n^B(0)$ and those defining $HCl_n(0)$. 
\begin{theorem}\label{thm:ch^B(Res)}
For each $I\subseteq[0,n-1]$, we have \[\ch^B(\Res_{H_n^B(0)}(\mathcal M_I))=\Delta^B([0,n-1]\setminus I).\]
\end{theorem}

\begin{proof}
Consider $1\leq k\leq 2^n$. Suppose $p\in\Peak^B([0,n-1]\setminus I)$; that is, $p-1\in I$ and $p\notin I$. If ${p\in K(k,I)}$, then it follows from \eqref{eq:K(k,I)} that $p\in D_k$, so $p-1\not\in K(k,I)$. On the other hand, if ${p\notin K(k,I)}$, then $p\not\in D_k$, so $p-1\in K(k,I)$. This shows that \[\Peak^B([0,n-1]\setminus I)\subseteq K(k,I)\triangle(K(k,I)+1).\] Also, note that if $0\notin I$, then $0\notin K(k,I)$ since $D_k\subseteq[n]$. Given a set $K\subseteq[0,n-1]$ such that $\Peak^B([0,n-1]\setminus I)\subseteq K\triangle(K+1)$ and $0\notin I\implies 0\notin K$, it follows from \Cref{lem:valleys} that the indices $k$ satisfying $K(k,I)=K$ are in bijection with subsets of $\Val^B([0,n-1]\setminus I)$. Hence, 
\begin{align*}
\ch^B(\Res_{H_n^B(0)}(\mathcal M_I))&=\sum_{k=1}^{2^n}\ch^B(\mathcal M_I^{(k)}/\mathcal M_I^{(k-1)}) \\ 
&=\sum_{k=1}^{2^n}\ch^B(\mathcal S_{K(k,I)}^B) \\ &=\sum_{k=1}^{2^n}F^B_{K(k,I)} \\ &=2^{|\Val^B([0,n-1]\setminus I)|}\sum_{\substack{\Peak^B([0,n-1]\setminus I)\subseteq K\triangle(K+1) \\ 0\notin I\implies 0\notin K}}F_K^B \\
&=2^{|\Peak^B([0,n-1]\setminus I)|+\zeta([0,n-1]\setminus I)}\sum_{\substack{\Peak^B([0,n-1]\setminus I)\subseteq K'\triangle(K'+1) \\ 0\in [0,n-1]\setminus I\implies 0\in K'}}F_{K'}^B,
\end{align*}
where we obtained the last equality by using the identity \eqref{eq:PeakVal} and substituting ${K'=[0,n-1]\setminus K}$. The desired result now follows immediately from the definition of $\Delta^B([0,n-1]\setminus I)$. 
\end{proof}

The proof of the following result is essentially the same as that of \cite[Theorem~5.4]{BHT}, but we include it anyway for the sake of completeness. 

\begin{theorem}
Suppose $I\subseteq[0,n-1]$, and let $V=\Val^B([0,n-1]\setminus I)$. Let $Cl_V$ be the subalgebra of $Cl_n$ generated by $\{c_v:v\in V\}$. For $c\in Cl_n$, define $f_c\colon\mathcal M_I\to\mathcal M_I$ by $f_c(x\varepsilon_I)=xc\varepsilon_I$ for all $x\in Cl_n$. We obtain a right action of $Cl_V$ on $\mathcal M_I$ by letting $(x\varepsilon_I)\cdot c=f_c(x\varepsilon_I)$ (and extending by linearity), and this right action commutes with the left action of $HCl_n^B(0)$. Moreover, the map $c\mapsto f_c$ is an isomorphism from $Cl_V$ to $\End_{HCl_n^B(0)}(\mathcal M_I)$. 
\end{theorem}

\begin{proof}
Because $\mathcal M_I$ is freely generated by $\varepsilon_I$ as a $Cl_n$-module, each of the endomorphisms in $\End_{HCl_n^B(0)}(\mathcal M_I)$ must be of the form $f_c$ for a unique $c\in Cl_n$. Therefore, it suffices to prove that $f_c\in\End_{HCl_n^B(0)}(\mathcal M_I)$ if and only if $c\in Cl_V$. Because every element of $HCl_n^B(0)$ can be written in the form $x\pi_{i_1}\cdots\pi_{i_\ell}$ for some $x\in Cl_n$ and some $i_1,\ldots,i_\ell\in[0,n-1]$, we have $f_c\in\End_{HCl_n^B(0)}(\mathcal M_I)$ if and only if 
\begin{equation}\label{eq:End}
\pi_i\cdot c\,\varepsilon_I=\begin{cases} -c\,\varepsilon_I & \mbox{ if } i\in I \\
0 & \mbox{ if } i\not\in I
\end{cases}
\end{equation}
for every $i\in[0,n-1]$. 

First, suppose $c=c_D$ for some $D\subseteq V$. In the ribbon diagram of $c_D\varepsilon_I$, no box with a barred entry has a box immediately above it or immediately to its right. For each $i\in[0,n-1]$, it follows from the above explicit description of the action of $\pi_i$ on $c_D\varepsilon_I$ that \[\pi_i\cdot c_D\varepsilon_I=\begin{cases} -c_D\varepsilon_I & \mbox{ if } i\in I \\
0 & \mbox{ if } i\not\in I.
\end{cases}\]
Extending by linearity, we find that $f_c\in\End_{HCl_n^B(0)}(\mathcal M_I)$ whenever $c\in Cl_V$. 

To prove the converse, suppose $c\not\in Cl_V$. Let $c_D\varepsilon_I$ be a term in the support of $c\,\varepsilon_I$ such that $D\not\subseteq V$; we may assume that this term is chosen maximally with respect to the partial order $\leq_I$. Because $D\not\subseteq V$, there exists $i$ such that either $i\in I$ and $i+1\in D$ or $i\not\in I$ and $i\in D$. Then either $i\in I$ and $\alpha(i,I,D)=0$ or $i\not\in I$ and $\alpha(i,I,D)=-1$. Combining \Cref{lem:triangular} with the $\leq_I$-maximality of $c_D\varepsilon_I$, we find that \eqref{eq:End} does not hold; thus, $f_c\not\in \End_{HCl_n^B(0)}(\mathcal M_I)$. 
\end{proof}

We now proceed to characterize when two $HCl_n^B(0)$-modules $\mathcal M_I$ and $\mathcal M_J$ are isomorphic. 

\begin{theorem}
Let $I,J\subseteq[0,n-1]$. The $HCl_n^B(0)$-modules $\mathcal M_I$ and $\mathcal M_J$ are isomorphic if and only if \[\Peak^B([0,n-1]\setminus I)=\Peak^B([0,n-1]\setminus J)\quad\text{and}\quad 0\not\in I\triangle J.\] 
\end{theorem}

\begin{proof}
If $\mathcal M_I$ and $\mathcal M_J$ are isomorphic, then $\ch^B(\Res_{H_n^B(0)}(\mathcal M_I))=\ch^B(\Res_{H_n^B(0)}(\mathcal M_J))$, so it follows from \Cref{thm:ch^B(Res)} that $\Peak^B([0,n-1]\setminus I)=\Peak^B([0,n-1]\setminus J)$ and $0\not\in I\triangle J$. 

To prove the converse, it suffices to prove that $\mathcal M_I\cong\mathcal M_J$ whenever the following hold: 
\begin{itemize}
\item $\Peak^B([0,n-1]\setminus I)=\Peak^B([0,n-1]\setminus J)$;
\item $0\notin I\triangle J$; 
\item $J=I\sqcup\{k\}$ for some $k\in[n-1]\setminus I$. 
\end{itemize}
Note that if $k+1\in I\cap J$, then $k+1$ is in $\Peak^B([0,n-1]\setminus J)$ but not in $\Peak^B([0,n-1]\setminus I)$, which is impossible. Thus, $k+1\notin I\cap J$. Similarly, if $k-1\in I\cup J$, then $k$ is in $\Peak^B([0,n-1]\setminus I)$ but not in $\Peak^B([0,n-1]\setminus J)$, which is impossible. Thus, $k-1\not\in I\cup J$. This implies that the pieces of the ribbon diagrams representing $\varepsilon_I$ and $\varepsilon_J$ near box $k$ look like  
\[\begin{array}{ccc}
\ddots &  &  \\
 & \hspace{-.2cm}\begin{ytableau}
{\scriptstyle k-1} & {\scriptstyle k} & {\scriptstyle k+1}\\
\none & \none & {\scriptstyle k+2}
\end{ytableau} &  \\
 &  & \hspace{-.17cm}\ddots \end{array} \qquad\text{and}\qquad \begin{array}{ccc}
\ddots &  &  \\
 & \hspace{-.2cm}\begin{ytableau}
{\scriptstyle k-1} & {\scriptstyle k}\\
\none & {\scriptstyle k+1}\\ 
\none & {\scriptstyle k+2}
\end{ytableau} &  \\
 &  & \hspace{-.17cm}\ddots \end{array},\] respectively. 
 
 Let $\eta=(c_kc_{k+1}-1)\varepsilon_J\in\mathcal M_J$, and define a linear map $f\colon \mathcal M_I\to\mathcal M_J$ by $f(x\varepsilon_I)=x\eta$ for all $x\in Cl_n$. The map $f$ is injective because $c_kc_{k-1}-1$ is invertible in $Cl_n$. Since $\mathcal M_I$ and $\mathcal M_J$ have the same dimension, $f$ is an isomorphism of $Cl_n$-modules. To show that $f$ is an isomorphism of $HCl_n^B(0)$-modules, we need to check that $\pi_i\cdot \eta=f(\pi_i\cdot\varepsilon_I)$ for all $i\in[0,n-1]$. We consider several cases. 

Suppose $i\notin\{k-1,k,k+1\}$. Then $\pi_i$ commutes with $c_kc_{k+1}-1$, so \[\pi_i\cdot\eta=(c_kc_{k+1}-1)\pi_i\cdot\varepsilon_J=f(\pi_i\cdot\varepsilon_I),\] as desired. 

Suppose $k=1$ and $i=0$. Then $0\notin I\cup J$, so \[\pi_0\cdot\eta=\pi_0(c_1c_2-1)\varepsilon_J=(\sqrt{-1}c_2\pi_0-\pi_0)\cdot\varepsilon_J=0=f(\pi_0\cdot\varepsilon_I),\] as desired. 

Suppose $k>1$ and $i=k-1$. Then $k-1\notin I\cup J$, so \[\pi_{k-1}\cdot\eta=(c_{k-1}c_{k+1}\pi_{k-1}-\pi_{k-1})\cdot\varepsilon_J=0=f(\pi_{k-1}\cdot \varepsilon_I),\] as desired. 

Suppose $i=k$. Then 
\begin{align*}
\pi_k\cdot\eta&=((c_{k+1}\pi_k+c_{k+1}-c_k)c_{k+1}-\pi_k)\cdot\varepsilon_J \\ 
&=(-c_kc_{k+1}\pi_k-1-c_kc_{k+1}-\pi_k)\cdot\varepsilon_J \\ 
&=-c_kc_{k+1}(-\varepsilon_J)-\varepsilon_J-c_kc_{k+1}\varepsilon_J-(-\varepsilon_J) \\ 
&=0 \\ 
&=f(\pi_k\cdot\varepsilon_I),
\end{align*} as desired. 

Finally, suppose $i=k+1$. Then $k+1\in I\cap J$, so 
\begin{align*}
\pi_{k+1}\cdot\eta&=(c_k(c_{k+2}\pi_{k+1}+c_{k+2}-c_{k+1})-\pi_{k+1})\cdot\varepsilon_J \\ 
&=(c_kc_{k+2}\pi_{k+1}+c_kc_{k+2}-c_kc_{k+1}-\pi_{k+1})\cdot\varepsilon_J \\ 
&=c_kc_{k+2}(-\varepsilon_J)+c_kc_{k+2}\varepsilon_J-c_kc_{k+1}\varepsilon_J-(-\varepsilon_J) \\ 
&=-\eta \\ 
&=f(\pi_k\cdot \varepsilon_I),
\end{align*} as desired. 
\end{proof} 

It would be interesting to determine the simple $HCl_n^B(0)$-modules. We note that the algebra $HCl_n^B(0)$ is not a superalgebra and that the approach used in \cite[Section 5.4]{BHT} for the simple $HCl_n(0)$-modules does not carry over.

\section{Type-$B$ Quasisymmetric Characteristics from $0$-Hecke--Clifford Modules}\label{sec:CliffordApplications}
In this section, we prove a result that allows us to compute the type-$B$ quasisymmetric characteristic of $\Res_{H_n^B(0)}\left(\Ind_{H_n^B(0)}^{HCl_n^B(0)}({\bf N})\right)$, where ${\bf N}$ is an $H_n^B(0)$-module of a special form. In particular, we will see that one can apply this theorem when ${\bf N}$ is a module obtained via the type-$B$ ascent compatibility framework. As a different application, we will provide a representation-theoretic interpretation of the shifted domino functions $H_\lambda$ from \cref{sec:domino}. 

Let $\mathscr X$ and $\mathscr Y$ be finite sets with $\mathscr X\subseteq \mathscr Y$. Let $\DD:\mathscr X\to2^{[0,n-1]}$ be a function, where $2^{[0,n-1]}$ is the power set of $[0,n-1]$. Let $f_0,\ldots,f_{n-1}:\mathscr X\to \mathscr Y$ be functions. Define linear operators $\pi_0,\ldots,\pi_{n-1}$ on $\mathbb C\mathscr X$ by letting \[\pi_i(y) = \begin{cases} -y & \mbox{ if } i\in\DD(y) \\
                            0 & \mbox{ if } i\notin\DD(y)\mbox{ and }f_i(y)\notin \mathscr X \\
                            f_i(y) & \mbox{ if } i\notin\DD(y)\mbox{ and }f_i(y)\in \mathscr X
\end{cases}\]
and extending by linearity. For $y,y'\in \mathscr X$, let us write $y\preceq y'$ if there is a (possibly empty) sequence $i_1,\ldots,i_r$ of indices in $[0,n-1]$ such that $y'=\pi_{i_r}\cdots\pi_{i_1}(y)$.

\begin{theorem}\label{thm:general_application}
Suppose the operators $\pi_0,\ldots,\pi_{n-1}$ define an action of $H_n^B(0)$ on $\mathbb C\mathscr X$. If the relation $\preceq$ is a partial order, then \[\ch^B\left(\Res_{H_n^B(0)}\left(\Ind_{H_n^B(0)}^{HCl_n^B(0)}(\mathbb C\mathscr X)\right)\right)=\sum_{y\in \mathscr X}\Delta^B([0,n-1]\setminus\DD(y)).\] 
\end{theorem}

\begin{proof}
A linear basis of $\Ind_{H_n^B(0)}^{HCl_n^B(0)}(\mathbb C\mathscr X)$ is $\{c_Dy:D\subseteq[0,n-1],\, y\in \mathscr X\}$. Let $y_1,\ldots,y_m$ be a linear extension of the poset $(\mathscr X,\preceq)$. For $0\leq k\leq m$, let $\widetilde{\bf N}_k=\text{span}\{c_Dy_j:D\subseteq[0,n-1],\, j\leq k\}$; then $\widetilde {\bf N}_k$ is an $HCl_n^B(0)$-module because $\preceq$ is a partial order. For $1\leq k\leq m$, a linear basis for the quotient module $\widetilde {\bf N}_k/\widetilde {\bf N}_{k-1}$ is $\{c_D\overline y_k:D\subseteq[0,n-1]\}$, where $\overline y_k=y_k+\widetilde {\bf N}_{k-1}$. We have \[\ch^B\left(\Res_{H_n^B(0)}\left(\Ind_{H_n^B(0)}^{HCl_n^B(0)}(\mathbb C\mathscr X)\right)\right)=\sum_{k=1}^m\ch^B\left(\Res_{H_n^B(0)}(\widetilde{\bf N}_k/\widetilde{\bf N}_{k-1})\right).\] Hence, the desired result will follow from \Cref{thm:ch^B(Res)} if we can prove that the $HCl_n^B(0)$-module $\widetilde {\bf N}_k/\widetilde {\bf N}_{k-1}$ is isomorphic to $\mathcal M_{\DD(y_k)}$. 

Define a map $\psi\colon\widetilde {\bf N}_k/\widetilde {\bf N}_{k-1}\to\mathcal M_{\DD(y_k)}$ by letting $\psi(c_D\overline y_k)=c_D\varepsilon_{\DD(y_k)}$ and extending by linearity. Then $\psi$ is clearly a linear isomorphism, so we just need to show that it respects the module structure. For $r\in[n]$ and $D\subseteq[n]$, let $p_r(D)=|[r]\cap D|$. We have \[c_r\psi(c_D\overline y_k)=c_rc_D\varepsilon_{\DD(y_k)}=(-1)^{p_r(D)}c_{D\triangle\{r\}}\varepsilon_{\DD(y_k)}=\psi((-1)^{p_r(D)}c_{D\triangle\{r\}}\overline y_k)=\psi(c_rc_D\overline y_k).\] This shows that $\psi$ commutes with the action of the generators $c_1,\ldots,c_n$. 

To complete the proof, we must show that $\pi_i\psi(c_D\overline y_k)=\psi(\pi_ic_D\overline y_k)$ for all $i\in[0,n-1]$ and $D\subseteq[n]$. Fix such an index $i$ and a subset $D$. Using the defining relations of $HCl_n^B(0)$, we can write $\pi_ic_D$ in the form $\sum_{E\subseteq[n]}c_E(\gamma_E+\delta_E\pi_i)$, where $\gamma_E,\delta_E\in\mathbb C$ for each $E\subseteq[n]$. If $i\in\DD(y_k)$, then we have $\pi_i\varepsilon_{\DD(y_k)}=-\varepsilon_{\DD(y_k)}=\psi(-\overline y_k)=\psi(\pi_i\overline y_k)$. If $i\not\in\DD(y_k)$, then  $\pi_i\varepsilon_{\DD(y_k)}=0$, and it follows from the definition of the partial order $\preceq$ that $\pi_iy_k\in\text{span}\{y_1,\ldots,y_{k-1}\}$ so that $\pi_i\overline y_k$ is $0$ in $\widetilde{\bf N}_k/\widetilde{\bf N}_{k-1}$. In either case, we have $\pi_i\varepsilon_{\DD(y_k)}=\psi(\pi_i\overline y_k)$. Hence, for each $E\subseteq [n]$, we have $c_E(\gamma_E+\delta_E\pi_i)\varepsilon_{\DD(y_k)}=\psi(c_E(\gamma_E+\delta_E\pi_i)\overline y_k)$. Thus, 
\begin{align*}
\pi_i\psi(c_D\overline y_k)&=\pi_ic_D\varepsilon_{\DD(y_k)} \\ &=\left(\sum_{E\subseteq[n]}c_E(\gamma_E+\delta_E\pi_i)\right)\varepsilon_{\DD(y_k)} \\ &=\psi\left(\left(\sum_{E\subseteq[n]}c_E(\gamma_E+\delta_E\pi_i)\right)\overline y_k\right) \\ &=\psi(\pi_ic_D\overline y_k).
\qedhere
\end{align*}
\end{proof}

\begin{corollary}
Let $X$ be an ascent-compatible subset of $B_n$, and let $\mathbb CX$ be the associated $H_n^B(0)$-module. Then \[\ch^B\left(\Res_{H_n^B(0)}\left(\Ind_{H_n^B(0)}^{HCl_n^B(0)}(\mathbb CX)\right)\right)=\sum_{x\in X}\Delta^B([0,n-1]\setminus\Des(x)).\] 
\end{corollary}

\begin{proof}
Simply apply \cref{thm:general_application} with $\mathscr X=X$, $\mathscr Y=B_n$, $f_i(x)=s_ix$, and $\DD(x)=\Des(x)$. 
\end{proof}

Let $\lambda$ be a partition whose $2$-quotient $(\mu,\nu)=((\mu_1, \ldots , \mu_p),(\nu_1, \ldots, \nu_q))$ satisfies $\mu_p\ge p$ and $\nu_q\ge q$. Let $\widehat\lambda$ denote the conjugate of $\lambda$. Recall from \cref{subsec:shifted} the definition of the set $\SShDT(\lambda)$ of standard shifted domino tableaux of shape $\lambda$. If we reflect a standard shifted domino tableau $Q$ of shape $\lambda$ across the main diagonal, we obtain a filled tiling $\widehat Q$ of $\widehat\lambda$ that we call a \dfn{conjugated standard shifted domino tableau} of shape $\widehat\lambda$. See \cref{fig:conjugated} for an example. Let $\widehat{\SShDT}(\widehat\lambda)$ denote the set of conjugated standard shifted domino tableaux of shape $\widehat\lambda$. 

\begin{figure}[ht]
\begin{center}\includegraphics[height=2.493cm]{HeckePIC3}\qquad\qquad\qquad\qquad \includegraphics[height=3.479cm]{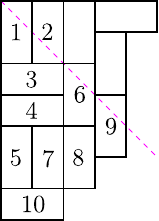}
  \end{center}
  \caption{A standard shifted domino tableau $Q$ of shape $\lambda=(7,7,6,5,1)$ (left) and its corresponding conjugated standard shifted domino tableau $\widehat Q$ of shape ${\widehat\lambda=(5,4,4,4,4,3,2)}$ (right).}\label{fig:conjugated}
\end{figure}

Descents of conjugated standard shifted domino tableaux are defined just as for standard shifted domino tableaux. Let $\widehat Q\in\widehat{\SShDT}(\widehat\lambda)$, and let $m$ be the number of dominoes weakly below the main diagonal in $\widehat Q$. We say $i\in [0,m-1]$ is a \dfn{descent} of $\widehat Q$ if $i=0$ and the domino in $\widehat Q$ containing $1$ is vertical or if $i>0$ and the domino in $\widehat Q$ containing $i+1$ is strictly lower than the domino in $\widehat Q$ containing $i$. Let $\Des(\widehat Q)$ be the set of descents of $\widehat Q$. Note that if $\widehat Q$ is the conjugate of the standard shifted domino tableau $Q\in\SShDT(\lambda)$, then $\Des(\widehat Q)=[0,n-1]\setminus\Des(Q)$. 

For $0\leq i\leq n-1$ and $Q\in\SShDT(\lambda)$, let $s_i(\widehat Q)=\widehat{s_i(Q)}$. Define operators $\pi_0, \pi_1, \ldots \pi_{n-1}$ on $\mathbb C\widehat{\SShDT}(\widehat\lambda)$ by
\[\pi_i(\widehat Q) = \begin{cases} -\widehat Q & \mbox{ if } i\in\Des(\widehat Q) \\
                            0 & \mbox{ if } i\notin\Des(\widehat Q)\mbox{ and }s_i(\widehat Q)\notin\widehat{\SShDT}(\widehat\lambda) \\
                            s_i(\widehat Q) & \mbox{ if } i\notin\Des(\widehat Q)\mbox{ and }s_i(\widehat Q)\in\widehat{\SShDT}(\widehat\lambda)
\end{cases}\] for each $Q\in\SShDT(\lambda)$. 

\begin{theorem}
Let $\lambda$ be a partition whose $2$-quotient $(\mu,\nu)=((\mu_1, \ldots , \mu_p),(\nu_1, \ldots, \nu_q))$ satisfies $\mu_p\ge p$ and $\nu_q\ge q$. The operators $\pi_0, \pi_1, \ldots , \pi_{n-1}$ define an action of $H_n^B(0)$ on $\mathbb C\widehat{\SShDT}(\widehat\lambda)$. Moreover, \[\ch^B\left(\Res_{H_n^B(0)}\left(\Ind_{H_n^B(0)}^{HCl_n^B(0)}(\mathbb C\widehat{\SShDT}(\widehat\lambda))\right)\right)=H_\lambda.\] 
\end{theorem}

\begin{proof}
The proof that $\pi_0, \pi_1, \ldots , \pi_{n-1}$ define an action of $H_n^B(0)$ is very similar to that of \cref{thm:dominoHecke}, so we omit it. To prove the desired identity involving the shifted domino function $H_\lambda$, we appeal to \cref{thm:general_application}. Let $\mathscr X=\widehat{\SShDT}(\widehat\lambda)$ and $\mathscr Y=\widehat{\SShDT}(\widehat\lambda)\cup\{\bullet\}$, where $\bullet$ is a formal symbol. For $0\leq i\leq n-1$, define $f_i\colon\mathscr X\to\mathscr Y$ by \[f_i(\widehat Q)=\begin{cases} s_i(\widehat Q) & \mbox{ if } s_i(\widehat Q)\in\widehat{\SShDT}(\widehat\lambda) \\
\bullet & \mbox{ if } s_i(\widehat Q)\notin\widehat{\SShDT}(\widehat\lambda).
\end{cases}\] Finally, let us define $\DD\colon\mathscr X\to 2^{[0,n-1]}$ by $\DD(\widehat Q)=\Des(\widehat Q)$. The desired result now follows from \cref{cor:H_lambda}, \cref{thm:general_application}, and the observation that $\Des(\widehat Q)=[0,n-1]\setminus\Des(Q)$ for each ${Q\in\SShDT(\lambda)}$. 
\end{proof}

\section*{Acknowledgments}
Colin Defant was supported by the National Science Foundation under Award No.\ 2201907 and by a Benjamin Peirce Fellowship at Harvard University. Dominic Searles was supported by the Marsden Fund, administered by the Royal Society of New Zealand Te Ap\=arangi.

\bibliography{TypeB}{}
\bibliographystyle{abbrv}

\end{document}